\documentclass[preprint,12pt]{elsarticle}
\usepackage[latin9]{inputenc}
\usepackage{mathtools}
\usepackage{enumitem}
\usepackage{amsmath}
\usepackage{amsthm}
\usepackage{amssymb}
\usepackage{cancel}
\usepackage{mathdots}
\usepackage{stmaryrd}
\usepackage{hyperref}

\makeatletter
\theoremstyle{plain}
\ifx\thechapter\undefined
	\newtheorem{thm}{\protect\theoremname}
\else
	\newtheorem{thm}{\protect\theoremname}[chapter]
\fi
\theoremstyle{remark}
\newtheorem{rem}[thm]{\protect\remarkname}
\theoremstyle{definition}
\newtheorem{example}[thm]{\protect\examplename}
\theoremstyle{plain}

\theoremstyle{plain}
\newtheorem{prop}[thm]{\protect\propositionname}
\theoremstyle{plain}
\newtheorem{lem}[thm]{\protect\lemmaname}

\usepackage{wasysym}
\usepackage{xcolor}
\usepackage{graphicx}

\theoremstyle{plain}

\providecommand{\assumptionname}{Assumption}

\providecommand{\examplename}{Example}
\providecommand{\lemmaname}{Lemma}
\providecommand{\propositionname}{Proposition}
\providecommand{\corollaryname}{Corollary}
\providecommand{\remarkname}{Remark}
\providecommand{\theoremname}{Theorem}

\DeclareMathOperator{\E}{{\mathbb E}}

\DeclareMathOperator{\D}{{\mathbb D}}
\DeclareMathOperator{\R}{{\mathbb R}}

\DeclareMathOperator{\N}{{\mathbb N}}

\DeclareMathOperator{\Q}{{\mathbb Q}}

\DeclareMathOperator{\dom}{dom}

\renewcommand{\phi}{\varphi}

\renewcommand{\theta}{\vartheta}
\renewcommand{\subset}{\subseteq}

\providecommand{\abs}[1]{\lvert #1 \rvert}
\providecommand{\norm}[1]{\lVert #1 \rVert}
\providecommand{\bnorm}[1]{{\Bigl\lVert #1 \Bigr\rVert}}

\providecommand{\wraum}{$(\Omega,{\scr F},\PP)$}
\providecommand{\fwraum}{$(\Omega,{\scr F},\PP,({\scr F}_t))$}

\newtheorem*{assumption*}{\assumptionnumber}\providecommand{\assumptionnumber}{}

\newenvironment{assumSigmaB}[3]{%
  \renewcommand{\assumptionnumber}{Assumption $#1$($#2$;$#3$)}%
  \begin{assumption*}%
 \protected@edef\@currentlabel{$#1$($#2$;$#3$)}%
 }{%
  \end{assumption*}
 }

\newcommand{\aswithargsref}[4]{\ref{#1}}

\makeatother

\providecommand{\corollaryname}{Corollary}
\providecommand{\examplename}{Example}
\providecommand{\lemmaname}{Lemma}
\providecommand{\propositionname}{Proposition}
\providecommand{\remarkname}{Remark}
\providecommand{\theoremname}{Theorem}

\begin{document}
\global\long\def\epsilon{\varepsilon}

\global\long\def\E{\mathbb{E}}

\global\long\def\I{\mathbf{1}}

\global\long\def\N{\mathbb{N}}

\global\long\def\R{\mathbb{R}}

\global\long\def\C{\mathbb{C}}

\global\long\def\Q{\mathbb{Q}}

\global\long\def\P{\mathbb{P}}

\global\long\def\D{\Delta_{n}}

\global\long\def\dom{\operatorname{dom}}

\global\long\def\b#1{\mathbb{#1}}

\global\long\def\c#1{\mathcal{#1}}

\global\long\def\s#1{{\scriptstyle #1}}

\global\long\def\u#1#2{\underset{#2}{\underbrace{#1}}}

\global\long\def\r#1{\xrightarrow{#1}}

\global\long\def\mr#1{\mathrel{\raisebox{-2pt}{\ensuremath{\xrightarrow{#1}}}}}

\global\long\def\t#1{\left.#1\right|}

\global\long\def\l#1{\left.#1\right|}

\global\long\def\f#1{\lfloor#1\rfloor}

\global\long\def\sc#1#2{\langle#1,#2\rangle}

\global\long\def\abs#1{\lvert#1\rvert}

\global\long\def\bnorm#1{\Bigl\lVert#1\Bigr\rVert}

\global\long\def\wraum{(\Omega,\c F,\P)}

\global\long\def\fwraum{(\Omega,\c F,\P,(\c F_{t}))}

\global\long\def\norm#1{\lVert#1\rVert}

\begin{frontmatter}


\author{Randolf Altmeyer}
\ead{ra591@maths.cam.ac.uk}
\address{University of Cambridge, United Kingdom}

\title{Central limit theorems for discretized occupation time functionals}

\begin{abstract}
The approximation of integral type functionals  is studied for discrete observations of a continuous Itô semimartingale. Based on novel  approximations in the Fourier domain, central limit theorems are proved for $L^2$-Sobolev functions  with fractional smoothness. An explicit $L^2(\P)$-lower bound shows  that already lower order quadrature rules, such as  the trapezoidal rule and the classical Riemann estimator, are rate optimal, but only the trapezoidal rule is efficient, achieving the minimal asymptotic variance.
\end{abstract}

\begin{keyword}
occupation time \sep semimartingale \sep integral functionals \sep Sobolev spaces \sep lower bound

\MSC 60F05 \sep 60G99 \sep 65D32

\end{keyword}

\end{frontmatter}

\section{Introduction}

For $T>0$ let $X=(X_{t})_{0\leq t\leq T}$ be an $\R^{d}$-valued
continuous Itô semimartingale on a filtered probability space $(\Omega,\c F,(\c F_{t})_{0\leq t\leq T},\P)$.
Consider the approximation of the \emph{occupation time functional}
\[
\Gamma_{t}\left(f\right)=\int_{0}^{t}f\left(X_{r}\right)dr,\,\,\,\,0\leq t\leq T,
\]
for a function $f$ from discrete observations $X_{t_{k}}$ at equidistant times
$t_{k}=k\Delta_{n}$, where $\Delta_{n}=T/n$ and $k\in\{0,\dots,n\}$.
This discretization problem appears naturally in numerical analysis
and statistics for stochastic processes. The mathematical challenge
is to determine an optimal approximation method and the rate at which
convergence takes place.  A canonical choice is the Riemann estimator
\[
\widehat{\Gamma}_{t,n}(f)=\Delta_{n}\sum_{k=1}^{\lfloor t/\D\rfloor}f(X_{t_{k-1}}),
\]
which has been analysed in several recent papers to obtain weak and strong $L^p(\P)$-approximations of $\Gamma_{t}(f)$ for non-smooth functions $f$ with rates of convergence depending on the properties of the process $X$ and on the regularity of $f$, cf.  \cite{Altmeyer2017, Altmeyer:2017dza,  Ganychenko2015, Kohatsu-Higa2014}. Central limit theorems for occupation and local times have been shown in \cite{ivanovs2022optimal}.  

For smooth $f$ more can be said, because then the process $t\mapsto f(X_t)$ is again a continuous Itô semimartingale.  It is well-known (e.g., \cite{gobet2007discrete}) that in this case the weak approximation error is of order $O(\Delta_n)$.  A central limit theorem for $\Delta_n^{-1}(\Gamma_{\Delta_n\lfloor t/\Delta_n\rfloor}(f)-\widehat{\Gamma}_{t,n}(f))$ was obtained in \citep[Chapter 6]{jacod2011discretization} with the weak limit depending only on $\nabla f$. This suggests that the central limit theorem might also hold for less smooth functions, but the proof of \citep{jacod2011discretization} relies crucially on Itô's formula and is therefore restricted to $f\in C^{2}(\R^{d})$.  

The goal of this work is to prove a central limit theorem for the Riemann estimator for general Itô processes in $\R^d$ and under minimal assumptions on the function $f$ such that $t\mapsto f(X_t)$
is not necessarily  a semimartingale.   Related to the classical work
of \citep{Geman1980} on occupation densities, the central idea is
to express the error $\Gamma_{\Delta_n\lfloor t/\Delta_n\rfloor}(f)-\widehat{\Gamma}_{t,n}(f)$ in
terms of the Fourier transform of $f$ and the complex exponentials
$e^{i\sc u{X_{t}}}$ for a frequency $u\in\R^d$.  The analysis of this requires approximating $X_t$ depending on $u$ and is inspired
by the one-step Euler approximations of \citep{Fournier:2010ed},  
applied here to a discretization problem different from the usual
piecewise constant approximation in time of semimartingales.  The approximation problem is therefore moved
to the frequency domain and regularity of $f$ is measured in the fractional $L^{2}$-Sobolev sense.  Under regularity assumptions on the coefficients of $X$ and assuming an additive perturbation by an independent random variable $\xi$ having bounded Lebesgue density we extend
the central limit theorem to $f\in H^{s}(\R^{d})$, $s\leq 2$, at 
the same rate $\Delta_{n}$.  If $X$ has independent increments, then this applies to $f\in H^1(\R^d)$.  The
proof ideas for the central limit theorem have also been applied in \citep{altmeyer2017fourier} to obtain generalized Itô formulas for functions $f$ with fractional Sobolev regularity.  We consider here only continuous Itô semimartingales, but extensions to more general processes including jumps seem possible.  

One might wonder if the rate of convergence $\Delta_{n}$ can be improved
using different estimators or quadrature rules.  From a probabilistic
point of view, a natural estimator is the conditional expectation
$\E[\Gamma_{t}(f)|\mathcal{G}_{n}]$, where $\c G_{n}=\sigma(X_{t_{k}}:k\in\{0,\dots,n\})$
is the sigma field generated by the data.  While there is generally no analytic expression for $\E[\Gamma_{t}(f)|\c G_{n}]$, it is a classical result in probabilistic numerics that it is given by the trapezoidal rule
\begin{align}
\widehat{\Theta}_{t,n}(f)=\Delta_{n}\sum_{k=1}^{\lfloor t/\Delta_{n}\rfloor}\frac{f(X_{t_{k-1}})+f(X_{t_{k}})}{2}\label{eq:trepzoidal}
\end{align}
if $f$ is the identity function and $X$ is a Brownian motion  \citep{diaconis1988bayesian}.  We show that the trapezoidal rule also satisfies a central limit theorem at the rate $\Delta_n$.  By proving an $L^2(\P)$ lower bound on the estimation error when $X$ is a Brownian motion we show that both $\widehat{\Gamma}_{t,n}(f)$ and  $\widehat{\Theta}_{t,n}(f)$ are rate-optimal and that the latter is also efficient in the sense that its asymptotic variance coincides with the minimal $L^2(\P)$ estimation error.  Related lower bounds for integral functionals for less smooth functions $f$ and local times have been obtained by \citep{Altmeyer2017, altmeyer2021optimal}.

The paper is organized as follows. In Section 2 we review the CLT for $f\in C^{2}(\R^{d})$ and extend it in Section 3 to $f$ with fractional Sobolev regularity.  Several special cases are studied to explore or relax the used assumptions. Section 4 presents the lower bound.  Proofs of the main results are deferred to Section 5.

Let us introduce some notation. $C$ always denotes a positive
absolute constant, which may change from line to line. We write $a\lesssim b$
for $a\leq Cb$ and $Y_{n}=o_{\P}(a_{n})$ if $a_{n}^{-1}Y_{n}\r{\P}0$
as $n\rightarrow\infty$ for a sequence of random variables $(Y_{n})_{n\geq1}$ and real numbers $(a_{n})_{n\geq1}$. If $Z^{(n)}$ and $Z$ are stochastic processes
on $[0,T]$, then $Z^{(n)}_{t}\r{ucp}Z_{t}$ means $\sup_{0\leq t\leq T}|(Z_{n})_{t}-Z_{t}|\r{\P}0$.  Stable convergence in law is denoted by $Z^{(n)}_t\r{st}Z_t$ and may refer, depending on the context,  to convergence at a fixed time $0\leq t\leq T$ or to functional convergence on the Skorokhod space $\c D([0,T],\R^{d})$.  For details on stable convergence the reader is referred to \citep{jacod2013limit}.

\section{\label{sec:CLT}Central limit theorems for $f\in C^2(R^d)$}

Recall (for example from \citep{jacod2011discretization}) that the Itô semimartingale $X$ can be realised as
\begin{equation}
X_{t}=X_{0}+\int_{0}^{t}b_{r}dr+\int_{0}^{t}\sigma_{r}dW_{r},\quad0\leq t\leq T,\label{eq:semimartingale equation for X}
\end{equation}
where $X_{0}$ is $\c F_{0}$-measurable, $(W_{t})_{0\leq t\leq T}$
is a standard $d$-dimensional Brownian motion, $b=(b_{t})_{0\leq t\leq T}$
is a locally bounded $\R^{d}$-valued process and $\sigma=(\sigma_{t})_{0\leq t\leq T}$
is a càdlàg $\R^{d\times d}$-valued process, all adapted to $(\c F_{t})_{0\leq t\leq T}$. 

For $f\in C^2(\R^d)$ the process $(f(X_t))_{0\leq t\leq T}$ is again a continuous Itô semimartingale. The following result is Theorem 6.1.2 in \citep{jacod2011discretization}, which itself is based on \citep{jacod2003asymptotic}. 
\begin{thm}
\label{thm:CLT_C2} For $f\in C^{2}(\R^{d})$ we have as $n\rightarrow\infty$ the stable convergence
\begin{equation}
\Delta_{n}^{-1}(\Gamma_{\Delta_n\lfloor t/\Delta_n\rfloor}(f)-\widehat{\Gamma}_{t,n}(f))\r{st}\frac{f(X_{t})-f(X_{0})}{2}+\frac{1}{\sqrt{12}}\int_{0}^{t}\sc{\nabla f(X_{r})}{\sigma_{r}d\widetilde{W}_{r}}\label{eq:stab_clt}
\end{equation}
as processes on the Skorokhod space $\c D([0,T],\R^{d})$,
where $\widetilde{W}$ is a $d$-dimensional Brownian motion, defined
on an independent extension of $(\Omega,\c F,(\c F_{t})_{0\leq t\leq T},\P)$.
\end{thm}

The proof of the CLT is based on the decomposition 
\begin{equation}
\Gamma_{\Delta_n\lfloor t/\Delta_n\rfloor}(f)-\widehat{\Gamma}_{t,n}(f)=M_{t,n}(f)+D_{t,n}(f)+E_{t,n}(f)\label{eq:decomposition}
\end{equation}
with
\begin{align*}
M_{t,n}(f) & =\sum_{k=1}^{\lfloor t/\Delta_{n}\rfloor}\int_{t_{k-1}}^{t_{k}}(f(X_{r})-\E[f(X_{r})|\c F_{t_{k-1}}])dr,\\
D_{t,n}(f) & =\sum_{k=1}^{\lfloor t/\Delta_{n}\rfloor}\int_{t_{k-1}}^{t_{k}}\E\left[f(X_{r})-f(X_{t_{k-1}})-\frac{f(X_{t_k})-f(X_{t_{k-1}})}{2}\bigg|\c F_{t_{k-1}}\right]dr,\\
E_{t,n}(f) & =\frac{\Delta_{n}}{2}\sum_{k=1}^{\lfloor t/\Delta_{n}\rfloor}\E[f(X_{t_{k}})-f(X_{t_{k-1}})|\c F_{t_{k-1}}].
\end{align*}
The proof proceeds by applying a standard CLT for triangular arrays of martingale differences (cf. Theorem IX.7.28 of [13]) to $(M_{t,n}(f))_{0\leq t\leq T}$, while the limit process of $(E_{t,n}(f))_{0\leq t\leq T}$ yields the asymptotic bias in \eqref{eq:stab_clt}.  For $f\in C^2(\R^d)$, $(D_{t,n}(f))_{0\leq t\leq T}$ is shown to be asymptotically negligible by Itô's formula.

\begin{rem}[Trapezoidal rule]
The \emph{trapezoidal rule} from \eqref{eq:trepzoidal} satisfies
\[
\widehat{\Theta}_{t,n}(f)=\widehat{\Gamma}_{t,n}(f)+\Delta_{n}\frac{f(X_{\lfloor t/\Delta_{n}\rfloor\Delta_{n}})-f(X_{0})}{2}.
\]
Since $f(X_{\lfloor t/\Delta_{n}\rfloor\Delta_{n}})$ converges uniformly
to $f(X_{t})$, the
CLT in Theorem \ref{thm:CLT_C2} provides us also with a functional
CLT for $\widehat{\Theta}_{t,n}(f)$:
\begin{equation}
\Delta_{n}^{-1}(\Gamma_{\Delta_n\lfloor t/\Delta_n\rfloor}(f)-\widehat{\Theta}_{t,n}(f))\r{st}\frac{1}{\sqrt{12}}\int_{0}^{t}\sc{\nabla f(X_{r})}{\sigma_{r}d\widetilde{W}_{r}}.\label{eq:trapezoid_clt}
\end{equation}
The trapezoidal rule achieves the same rate as the
Riemann estimator, but is asymptotically unbiased, as opposed to (\ref{eq:stab_clt}).
Moreover,  Theorem \ref{thm:lowerBound_H1} 
below shows for a Brownian motion $X$ that the trapezoidal rule is \emph{efficient} in the sense that it attains the minimal asymptotic variance among all square-integrable estimators for $\Gamma_t(f)$ from the observations $X_{t_k}$. For simplicity, we consider
in the following only $\widehat{\Gamma}_{t,n}(f)$, but results transfer
to $\widehat{\Theta}_{t,n}(f)$. 
\end{rem}

\section{\label{sec:CLT_Hs}Central limit theorems for $f\in H^s(R^d)$} 

If $f$ is not smooth,  then $(f(X_t))_{0\leq t\leq T}$ is generally not a semimartingale and it is not clear if the strategy from the last section still applies to prove a CLT.  Inspired by the observation that the limit process in (\ref{eq:stab_clt}) requires formally only a (weak) derivative for $f$, we aim now at deriving a CLT for functions in the fractional Sobolev space of regularity $s\geq0$
\begin{align*}
	H^s(\R^d) &= \{f:\norm{f}_{L^2} + \norm{f}_{H^s}<\infty\},\quad \norm{f}^2_{H^s} = \int_{\R^{d}}\left|\c Ff\left(u\right)\right|^{2}\left|u\right|^{2s}du,
\end{align*}
where $\c Ff$ is the Fourier transform, which for $f\in L^{1}(\R^{d})\cap L^{2}(\R^{d})$ is defined as $\c Ff(u)=\int_{\R^{d}}f(x)e^{i\sc ux}dx$, $u\in\R^{d}$.   We further say that $f\in H_{loc}^{s}(\R^{d})$, if $f\cdot\varphi\in H^{s}(\R^{d})$
for all smooth and compactly supported $\varphi\in C_{c}^{\infty}(\R^{d})$. It is well-known that
$C^{k}(\R^{d})\subset H_{loc}^{k}(\R^{d})$ for $k\in\N$ and the Sobolev embedding implies $H_{loc}^{s}(\R^{d})\subset C^{k}(\R^{d})$ if $s>d/2+k$,  cf. Section 2.7 of \citep{triebel2010theory}. 

A key assumption in this section is to consider instead of $X$ the process $X+\xi$ with an independent random variable $\xi$.   $L^2(\P)$ bounds on the terms in \eqref{eq:decomposition} will be obtained from the following simple lemma. 

\begin{lem}\label{lem:plancherel} Let $\xi$ be a random
variable, independent of the filtration $\c F$ with bounded Lebesgue
density. Suppose that $h(x)\equiv h(X,x)$, $x\in\R^d$, is a family of random variables such that $x\mapsto h(x)$ is Borel-measurable and such that $\P$-almost surely $h \in L^2(\R^d)$. Then $$\E[h^2(\xi)] \lesssim \int_{\R^d} \E[|h(u)|^2] du = (2\pi)^{-d} \int_{\R^d} \E[|\mathcal{F}h(u)|^2] du.$$
\end{lem}

\begin{proof}
By the independence of $X$ and $\xi$ we have 
\begin{align*}
	\E[h^2(\xi)] = \E\left[\E[h^2(\xi)|X]\right] \lesssim \E\left[ \int_{\R^d} h^2(x) dx\right]=\int_{\R^d} \E[|h(u)|^2] du,\quad u\in\R^d.
\end{align*}
The claim follows therefore from the Plancherel theorem.
\end{proof}
Introducing the Fourier transform in this way,  combined with the shift property of the Fourier transform $$\c Ff(X_t+\cdot)(u) = \c Ff(u)e^{-i\langle u,X_t\rangle},$$allows for separating the function $f$ from the process $X$.  This decomposition leads naturally to an analysis depending on the fractional Sobolev regularity of $f$ and on the characteristic function of the marginals of $X$.  Since the latter are usually not known explicitly,  $X$ is approximated depending on the frequency $u$.  For this we make the following assumptions, cf. \citep{Jacod:2015gx}.

\begin{assumSigmaB}{S}{\alpha}{\beta}\label{assu:sigmaB} Let $0\leq\alpha,\beta\leq1$.
There exists an increasing sequence of stopping times $(\tau_{R})_{R\geq1}$
with $\tau_{R}\rightarrow\infty$ for $R\rightarrow\infty$ such that
for all $0\leq s,t\leq T$ with $t+s\leq T$ 
\begin{align*}
\E\left[\sup_{0\leq r\leq s}\left|\sigma_{\left(t+r\right)\wedge\tau_{R}}-\sigma_{t\wedge\tau_{R}}\right|^{2}\right] & \leq Cs^{2\alpha},\,\,\,\,\E\left[\sup_{0\leq r\leq s}\left|b_{\left(t+r\right)\wedge\tau_{R}}-b_{t\wedge\tau_{R}}\right|^{2}\right]\leq Cs^{2\beta}.
\end{align*}
Moreover, $\sup_{0\leq t\leq T}\left|(\sigma_{t}\sigma_{t}^{\top})^{-1}\right|<\infty$
$\P$-almost surely.
\end{assumSigmaB}

The non-degeneracy of $\sigma_{t}\sigma_{t}^{\top}$ is a technical
condition and can probably be relaxed (it
is not necessary in Theorem \ref{thm:CLT_C2}).  For Itô semimartingales $b$ and $\sigma$ we have $\alpha=\beta=1/2$,  which also allows for non-predictable jumps.  Other
important examples are $\sigma$ and $b$ with $\alpha$- and $\beta$-Hölder continuous paths and with integrable Hölder constants, for instance with $\alpha<H$ when $\sigma$ is a fractional Brownian motion of Hurst index $0<H<1$. 

We can now formulate our first main result.
\begin{thm}
\label{thm:CLT_Hs} Let $s\geq 1$ and grant Assumption \aswithargsref{assu:sigmaB}{S}{\alpha}{\beta} with
$\alpha>\max(0,1-s/2)$, $\beta>0$.  Let $\xi$ be a random
variable, independent of the filtration $\c F$ with bounded Lebesgue
density. Then we have for $0\leq t\leq T$ and $f\in H^{s}(\R^{d})$
(or $f\in H_{loc}^{s}(\R^{d})$ and $\xi$ bounded) the stable convergence
\begin{align*}
 & \Delta_{n}^{-1}\left(\Gamma_{\Delta_n\lfloor t/\Delta_n\rfloor}\left(f(\cdot+\xi)\right)-\widehat{\Gamma}_{t,n}\left(f(\cdot+\xi)\right)\right)\\
 & \qquad\r{st}\frac{f(X_{t}+\xi)-f(X_{0}+\xi)}{2}+\frac{1}{\sqrt{12}}\int_{0}^{t}\sc{\nabla f(X_{r}+\xi)}{\sigma_{r}d\widetilde{W}_{r}},
\end{align*}
as $n\rightarrow\infty$, where $\widetilde{W}$ is as in Theorem \ref{thm:CLT_C2}.
\end{thm}
We conclude that the Riemann estimator satisfies the CLT  in Theorem  \ref{thm:CLT_C2}  at the optimal rate $\Delta_{n}$ for Sobolev-smooth functions.  Compared to Theorem \ref{thm:CLT_C2}, here the stable convergence holds only at a fixed time $t$, because it is difficult to control the term $D_{t,n}(f)$ in \eqref{eq:decomposition}  uniformly in $t$. It is unclear if this can be achieved for Sobolev functions,  in general.  
Note that the compositions $\nabla f(X_r+\xi)$ are well-defined random variables for $f\in H^1(\R^d)$, because by independence $X_t+\xi$ has a Lebesgue density and so $\nabla f=\nabla \tilde{f}$ almost surely implies $\nabla f(X_r+\xi)=\tilde{f}(X_r+\xi)$ $\P$-almost surely.
 
\begin{rem}[Regularisation by $\xi$]
The additive perturbation by $\xi$ is required in the proofs and is essential to our approach of weakening the regularity conditions on $f$.  
This is conceptually related to the averaging
by noise phenomenon \citep{Catellier2016} by regularising the underlying discretization problem through convolution smoothing.  Indeed,  as discussed above  $\mu$ is the Lebesgue density of $\xi$ and $X_t$ has marginal density $p_t$, then $X_t + \xi$ has density $p_t*\mu$, where $*$ denotes the convolution operator.  Alternatively, $f(X_t+\xi)$ corresponds in average to $f*\mu$.  These two different points of views have been explored in \cite{Altmeyer2017} for the approximation of occupation time functionals and $f\in H^s(\R^d)$, $0\leq s\leq 1$.
\end{rem}

\begin{rem}[Regularity of $X$ and $f$]
Theorem \ref{thm:CLT_Hs} presents a trade-off between the regularity of $X$ and $f\in H^s(\R^d)$.  For $s\leq 2$ the CLT applies as soon as $\beta>0$ and $\alpha + s/2>1$.  This means, the more regular $\sigma$ is, the less regular $f$ can be.  If $\sigma$ is a semimartingale
(and thus $\alpha=1/2$), then we need only $s>1$, while for $s>2$ any $\alpha$ is admissible.  
\end{rem}

If the characteristic functions of the $X_t$ are known explicitly,  then regularity conditions in the CLT and its proof simplify.  For example, for $X$ with independent increments (and thus with deterministic $b$ and $\sigma$) independence of $X_0=\xi$ and $X-X_0$ is trivially true.  In this case we only need to require $b$ and $\sigma$ to be càdlàg functions and the result applies to any $f\in H^1(\R^d)$.

\begin{thm}
\label{thm:CLT_II} Suppose that $X_{0}$ has a bounded Lebesgue density.  Assume that $b,\sigma$ are deterministic càdlàg functions and that $\sup_{0\leq t\leq T}|(\sigma_{t}\sigma_{t}^{\top})^{-1}|<\infty$.
Then we have for $0\leq t\leq T$ and $f\in H^{1}(\R^{d})$ the stable convergence
\begin{align*}
 & \Delta_{n}^{-1}\left(\Gamma_{\Delta_n\lfloor t/\Delta_n\rfloor}\left(f\right)-\widehat{\Gamma}_{t,n}\left(f\right)\right)\r{st}\frac{f(X_t)-f(X_{0})}{2}+\frac{1}{\sqrt{12}}\int_{0}^{t}\sc{\nabla f(X_r)}{\sigma_{r}d\widetilde{W}_{r}},
\end{align*}
as $n\rightarrow\infty$, where $\widetilde{W}$ is as in Theorem \ref{thm:CLT_C2}.
\end{thm}
Since $X_r$ has a Lebesgue density (e.g., by \cite{Romito2017a}), we can argue as after Theorem \ref{thm:CLT_Hs} that $\nabla f(X_r)$ is a well-defined random variable for $f\in H^1(\R^d)$ which depends only on the equivalence class of $f$ and not on the chosen representative.  In dimension $d=1$ the condition on $X_0$ can be removed.
\begin{thm}
\label{thm:CLT_III} If $d=1$, then Theorem \ref{thm:CLT_II} applies to any initial value $X_0$.
\end{thm}

Next,  for integrable $f$ with integrable Fourier transform we obtain a CLT with functional convergence and without $\xi$.  For this define the Fourier Lebesgue spaces of regularity $s\geq 0$
\begin{align*}
FL^s(\R^d) &= \{f:\norm{f}_{L^1} + \norm{f}_{FL^s}<\infty\},\quad \norm{f}_{FL^s} = \int_{\R^{d}}\left|\c Ff\left(u\right)\right|\left|u\right|^{s}du.
\end{align*}

We further say that $f\in FL_{loc}^{s}(\R^{d})$, if $f\cdot\varphi\in FL^{s}(\R^{d})$ for all $\varphi\in C_{c}^{\infty}(\R^{d})$.  

\begin{example}
If $f\in H^{s}(\R^{d})$ and
$\mu\in L^{2}(\R^{d})$, then $f*\mu\in FL^{s}(\R^{d})$.  Moreover, if $f\in H_{loc}^{s'}(\R^{d})$ for $s'>s+d/2$, then $f\in FL_{loc}^{s}(\R^{d})$.
\end{example}

By the Fourier inversion formula it follows $FL^{1}(\R^d)\subset C^1(\R^d)$.  In this sense the next theorem generalises Theorem \ref{thm:CLT_C2} without requiring $f\in C^2(\R^d)$.

\begin{thm}
\label{thm:CLT_Fs} Let $s\geq 1$ and grant Assumption \aswithargsref{assu:sigmaB}{S}{\alpha}{\beta} with
$\alpha>\max(0,1-s/2)$, $\beta>0$.  Then the functional stable
convergence in (\ref{eq:stab_clt}) holds for any $f\in FL_{loc}^{s}(\R^{d})$.
\end{thm}

\section{\label{sec:Lower-bounds}Optimality for Brownian motion}

We study
next the optimality in the $L^{2}(\P)$-sense for estimating $\Gamma_{T}(f)$
at the fixed time $T$ for a given $f\in H^{1}(\R^{d})$. 

Considering for $\Gamma_{T}(f)$ all square-integrable estimators,
which are measurable with respect to the
sigma field $\c G_{n}=\sigma(X_{t_{k}}:0\leq k\leq n)$,   shows that the minimal error is
\[
\inf_{\widehat{\Gamma}}\norm{\Gamma_{T}(f)-\widehat{\Gamma}}_{L^{2}(\P)}=\norm{\Gamma_{T}(f)-\widehat{\Gamma}^{*}}_{L^{2}(\P)},
\]
which is attained by the conditional expectation
$\widehat{\Gamma}^{*}=\E[\Gamma_{T}(f)|\mathcal{G}_{n}]$.  The asymptotic estimation error can be computed explicitly when $X$ is a Brownian motion.
\begin{thm}
\label{thm:lowerBound_H1}Let $X$ be a Brownian motion and let $f\in H^{1}(\R^{d})$.
Suppose that $X_{0}$ has a bounded Lebesgue density or $d=1$.  
Then
\begin{align*}
\lim_{n\rightarrow\infty}\left(\D^{-1}\norm{\Gamma_{T}(f)-\widehat{\Gamma}^{*}}_{L^{2}(\P)}\right) & =\E\left[\frac{1}{12}\int_{0}^{T}|\nabla f(X_{t})|^{2}dt\right]^{1/2}.
\end{align*}
\end{thm}

In view of Theorems \ref{thm:CLT_II} and \ref{thm:CLT_III}, this means that both $\widehat{\Gamma}_{T,n}(f)$ and the trapezoidal rule estimator $\widehat{\Theta}_{T,n}(f)$
are rate optimal for $f\in H^{1}(\R^{d})$ when $X$ is a Brownian
motion, while $\widehat{\Theta}_{T,n}(f)$ is even efficient. In particular,  no other quadrature rule for equidistant observation times can achieve a better rate
by exploiting higher smoothness of $f$.  The minimal asymptotic
$L^2(\P)$ error corresponds exactly to the asymptotic variance from (\ref{eq:trapezoid_clt})
with respect to $\widehat{\Theta}_{T,n}(f)$.  

\section{Proofs}

In the following we rely on well-known properties of the Fourier transform, cf. \cite{adams2003sobolev}. For example,  for $a\in\R$
\begin{align*}
	h\in L^2(\R^d): & \quad \mathcal{F}h(a+\cdot)(u) = \mathcal{F}h(u)e^{-i\langle u,a\rangle},\\
	h\in H^1(\R^d): & \quad \mathcal{F}(\partial_j h)(u) = iu_j \mathcal{F}h(u),\quad j=1,\dots,d.
\end{align*}

\subsection{\label{subsec:Proof-of-Theorem-CLT_Hs}Proof of Theorem \ref{thm:CLT_Hs}}

Throughout write $Y=X+\xi$, where $\xi$ is independent of $X$ and has a bounded Lebesgue density. 

\subsubsection{\label{subsubsec:localization}Localization}

By a well-known localization procedure (cf. Lemma 4.4.9 in \citep{jacod2011discretization})
and Assumption \aswithargsref{assu:sigmaB}{S}{\alpha}{\beta}, it suffices
to prove the CLT under the following stronger Assumption. 

\begin{assumSigmaB}{S_{loc}}{\alpha}{\beta}\label{assu:sigmaB_local}Let
$0\leq \alpha,\beta\leq1$. Then it holds $\P$-a.s.
\[
\sup_{0\leq t\leq T}\left(\left|X_{t}\right|+\left|b_{t}\right|+\left|\sigma_{t}\right|+\left|(\sigma_{t}\sigma_{t}^{\top})^{-1}\right|\right)\leq C,
\]
and for all $0\leq t,t'\leq T$ with $t+t'\leq T$
\begin{align*}
\E\left[\sup_{0\leq r\leq t'}\left|\sigma_{t+r}-\sigma_{t}\right|^{2}\right] & \leq C (t')^{2\alpha},\,\,\,\,\E\left[\sup_{0\leq r\leq t'}\left|b_{t+r}-b_{t}\right|^{2}\right]\leq C (t')^{2\beta}.
\end{align*}
\end{assumSigmaB}

When $f\in H_{loc}^{s}(\R^{d})$ and $\xi$ is bounded, this assumption
allows us to reduce the argument to $f\in H^{s}(\R^{d})$. Indeed,
let $\varphi$ be a smooth function with $\varphi=1$ on $B_{C+C_{\xi}}=\{x\in\R^{d}:|x|\leq C+C_{\xi}\}$,
where $|\xi|\leq C_{\xi}$ for a constant $C_{\xi}$, and with compact
support in $B_{C+C_{\xi}+\epsilon}$, $\epsilon>0$. If $f\in H_{loc}^{s}(\R^{d})$,
then $\tilde{f}=f\varphi\in H^{s}(\R^{d})$ and $\Gamma_{t}(f(\cdot+\xi))=\Gamma_{t}(\tilde{f}(\cdot+\xi))$,
$\widehat{\Gamma}_{t,n}(f(\cdot+\xi))=\widehat{\Gamma}_{t,n}(\tilde{f}(\cdot+\xi))$.

\subsubsection{Approximation results}

In this section we prove some useful approximation results for the process $X$.  For $t,\epsilon>0$ let $\lfloor t\rfloor_{\epsilon}=\lfloor t/\epsilon\rfloor\epsilon$,
$t(\epsilon)=\max(\lfloor t\rfloor_{\epsilon}-\epsilon,0)$. $t(\epsilon)$
projects $t$ onto the grid $\{0,\epsilon,2\epsilon,\dots,\lceil T/\epsilon\rceil\epsilon\}$
such that $t-t(\epsilon)\leq2\epsilon$ and $t-t(\epsilon)\geq\epsilon\wedge t$.
Set
\begin{align}
\widetilde{X}_{t}(t') & =X_{t'}+b_{t'}(t-t')+\sigma_{t'}(W_{t}-W_{t'}),\,\,\,0\leq t'\leq t.\label{eq:X_approx}
\end{align}

\begin{lem}
\label{lem:estimatesForSM} Grant Assumption \aswithargsref{assu:sigmaB_local}{S_{loc}}{\alpha}{\beta}.
Then we have $\P$-a.s. for all $0\leq r,t,t'\leq T$, $t+r\leq t',T$
\begin{enumerate}
\item $\E[\sup_{0\leq r'\leq r}|X_{t+r'}-X_{t}|^{p}|\c F_{t}]^{1/p}\lesssim r^{1/2}$
for $p\geq1$,
\item $\E[\sup_{0\leq r'\leq r}|X_{t+r'}-\widetilde{X}_{t+r'}(t)|^{2}|\c F_{t}]\lesssim r^{2\beta+2}+r^{2\alpha+1}$,
\item $\E[\sup_{0\leq r'\leq r}|\widetilde{X}_{t'}(t+r')-\widetilde{X}_{t'}(t)|^{2}]\lesssim r+r^{2\beta}+r^{2\alpha}$.
\end{enumerate}
\end{lem}

\begin{proof}
The first two results follow from the conditional Burkholder-Davis-Gundy
inequality, applied componentwise, cf. \citep[Section 2.1.5]{jacod2011discretization},
and from Assumption \aswithargsref{assu:sigmaB_local}{S_{loc}}{\alpha}{\beta}.
For (iii) write
\begin{align*}
 & \widetilde{X}_{t'}(t+r')-\widetilde{X}_{t'}(t)=X_{t+r'}-X_{t}+(b_{t+r'}-b_{t})(t'-(t+r'))\\
 & \qquad+(\sigma_{t+r'}-\sigma_{t})(W_{t'}-W_{t+r'})-b_{t}r'-\sigma_{t}(W_{t+r'}-W_{t}),
\end{align*}
and conclude by (i) and again Assumption \aswithargsref{assu:sigmaB_local}{S_{loc}}{\alpha}{\beta}.
\end{proof}
\begin{lem}
\label{lem:estimates_f} Grant Assumption \aswithargsref{assu:sigmaB_local}{S_{loc}}{0}{0}. 
Then the following holds for $f\in H^{1}(\R^{d})$ and as $n\rightarrow\infty$:
\begin{enumerate}
\item $\Delta_{n}\sum_{k=1}^{n}\E[|\nabla f(Y_{t_{k-1}})|^{2}]=O(1)$,
\item $\sum_{k=1}^{n}\int_{t_{k-1}}^{t_{k}}\E[\langle\nabla f(Y_{t_{k-1}}),Z_{r}\rangle^{2}]dr = O(\Delta_n)$
for $Z_{r}=X_{r}-X_{\lfloor r/\Delta_{n}\rfloor\Delta_{n}}$ and $Z_{r}=X_{r}-\widetilde{X}_{r}(\lfloor r/\Delta_{n}\rfloor\Delta_{n})$,
\item $\sum_{k=1}^{n}\int_{t_{k-1}}^{t_{k}}\E[(f(Z_{r}+Y_{t_{k-1}})-f(Y_{t_{k-1}})-\langle\nabla f(Y_{t_{k-1}}),Z_{r}\rangle)^{2}]dr$$=o(\Delta_{n})$
for $Z_{r}=X_{r}-X_{\lfloor r/\Delta_{n}\rfloor\Delta_{n}}$ and $Z_{r}=X_{\lceil r/\Delta_{n}\rceil\Delta_{n}}-X_{\lfloor r/\Delta_{n}\rfloor\Delta_{n}}$,
\item $\sum_{k=1}^{n}\int_{t_{k-1}}^{t_{k}}\E[|\gamma_{r}^{\top}\nabla f(Y_{r})-\gamma_{t_{k-1}}^{\top}\nabla f(Y_{t_{k-1}})|^2]dr=o(1)$, where $(\gamma_r)_{0\leq r\leq T}$ is a càdlàg process such that $\P$-a.s. $\sup_{0\leq r\leq T}|\gamma_r|\lesssim 1$,
\item $\sup_{0\leq t\leq T}\E[|f(Y_{\lfloor t/\Delta_{n}\rfloor\Delta_{n}})-f(Y_{t})|^{2}]=o(1)$.
\end{enumerate}
\end{lem}

\begin{proof}
(i).  Write $\nabla f(Y_{t_{k-1}})=h(\xi)$ with $h(x)=\nabla f(X_{t_{k-1}} + x)$ such that $|\mathcal{F}h(u)|^2=|\mathcal{F}f(u)|^2 |u|^2|e^{-i\langle u,X_{t_{k-1}}\rangle}|^2 = |\mathcal{F}f(u)|^2 |u|^2$.  The claim follows from Lemma \ref{lem:plancherel}.

(ii).  Let $t_{k-1}\leq r<t_{k}$ and write $\langle\nabla f(X_{t_{k-1}}+x),Z_{r}\rangle=h(\xi)$ with $h(x)=\langle\nabla f(X_{t_{k-1}}+x),Z_{r}\rangle$ such that $\mathcal{F}h(u) = \mathcal{F}f(u) \langle iu,Z_r\rangle e^{-i\langle u,X_{t_{k-1}}\rangle}$.  As in (i) we get from Lemma \ref{lem:plancherel} that
\begin{align*}
 & \E[\langle\nabla f(Y_{t_{k-1}}),Z_{r}\rangle^{2}]\lesssim\norm f_{H^{1}}^{2}\E\left[\left|Z_{r}\right|^{2}\right].
\end{align*}
The result follows from $\sup_{r}\E|Z_r|^2\lesssim \Delta_n$ using Lemma \ref{lem:estimatesForSM}(i,ii).

(iii).  Let $t_{k-1}\leq r<t_{k}$ and take
\begin{align*}
	h(x) &= f(Z_{r}+X_{t_{k-1}}+x)-f(X_{t_{k-1}}+x)-\langle\nabla f(X_{t_{k-1}}+x),Z_{r}\rangle\\
	&= \int_0^1 \langle \nabla f(X_{t_{k-1}}+x+aZ_{r})-\nabla f(X_{t_{k-1}}+x), Z_r\rangle da, 
\end{align*}
implying 
\begin{align*}
	|\mathcal{F}h(u)|=\left|\mathcal{F}f(u) \int_0^1 \langle iu (e^{-i\langle u,aZ_r\rangle} - 1),Z_r\rangle da\right|.
\end{align*}
Lemma  \ref{lem:plancherel} gives
\begin{align*}
 & \E[(f(Z_{r}+Y_{t_{k-1}})-f(Y_{t_{k-1}})-\langle\nabla f(Y_{t_{k-1}}),Z_{r}\rangle)^{2}] = \E[|h(\xi)|^2 |] \\
  & \quad\lesssim \int_{\R^{d}} |\c Ff(u)|^{2}|u|^{2} \E[\int_{0}^{1}|e^{-i\sc u{aZ_{r}}}-1|^{2}da |Z_{r}|^{2}]du\\
  & \quad \lesssim \Delta_n \int_{\R^{d}} |\c Ff(u)|^{2}|u|^{2} \E[\int_{0}^{1}|e^{-i\sc u{aZ_{r}}}-1|^{4}da]^{1/2} du, 
\end{align*}
using in the last line the Cauchy-Schwarz inequality and that $\E[|Z_{r}|^{4}] \lesssim \Delta^2_n$ uniformly in $r$ by Lemma \ref{lem:estimatesForSM}(i).  Now, observe that uniformly in $n$ and $r$, $\E[\int_{0}^{1}|e^{-i\sc u{aZ_{r}}}-1|^{4}da]\lesssim 1$ and again that by Lemma  \ref{lem:estimatesForSM}(i) uniformly in $r$ pointwise for $u\in\R^d$ as $n\rightarrow\infty$
\begin{align*}
	\E[\int_{0}^{1}|e^{-i\sc u{aZ_{r}}}-1|^{4}da]\lesssim |u|^4 \E[|Z_r|^4] \rightarrow 0.
\end{align*}
The claim follows therefore from dominated convergence.

(iv).  Let again $t_{k-1}\leq r<t_{k}$ and take this time $h(x)=\gamma_{r}^{\top}\nabla f(X_{r}+x)-\gamma_{t_{k-1}}^{\top}\nabla f(X_{t_{k-1}}+x)$ such that 
$$\c Fh(u)=\c Ff(u) ig_r(u)^\top u,\quad g_r(u)=\gamma_r e^{-i\langle u,X_r\rangle} - \gamma_{t_{k-1}} e^{-i\langle u,X_{t_{k-1}}\rangle}.$$ Lemma  \ref{lem:plancherel} shows 
\begin{align*}
 & \E[|\gamma_{r}^{\top}\nabla f(Y_{r})-\gamma_{t_{k-1}}^{\top}\nabla f(Y_{t_{k-1}})|^{2}] = \E[|h^2(\xi)|^2] \lesssim\int_{\R^{d}}\left|\c Ff(u)\right|^{2}\left|u\right|^{2}\E[|g_{r}(u)|^2]du.
\end{align*}
Here,  $|g_r(u)|$ is almost surely bounded.  Moreover, for all $u\in\R^d$ and uniformly in $r$ we get  $\E[|g_{r}(u)|^2]\rightarrow 0$ by dominated convergence, as both $\gamma$ and $X$ have càdlàg paths.  Another application of dominated convergence yields (iv). The proof of (v) is analogous to (iv) and is therefore skipped.
\end{proof}

\subsubsection{\label{subsubsec:mainDecomp}The main decomposition}

We apply the decomposition  from \eqref{eq:decomposition}, but with $X$ replaced by $Y$ in the definitions of $M_{t,n}(f)$,  $E_{t,n}(f)$ and $D_{t,n}(f)$. Observe first the following two propositions.

\begin{prop}
\label{prop:M_convergence_H1} Grant Assumption \aswithargsref{assu:sigmaB_local}{S_{loc}}{0}{0}.
Then we have for $f\in H{}^{1}(\R^{d})$ with $\widetilde{W}$ as
in Theorem \ref{thm:CLT_C2} the functional stable convergence 
\begin{equation*}
\Delta_{n}^{-1}M_{t,n}\left(f\right)\stackrel{st}{\to}\frac{1}{2}\int_{0}^{t}\sc{\nabla f(Y_{r})}{\sigma_{r}dW_{r}}+\frac{1}{\sqrt{12}}\int_{0}^{t}\sc{\nabla f(Y_{r})}{\sigma_{r}d\widetilde{W}{}_{r}}.
\end{equation*}
\end{prop}

\begin{prop}
\label{prop:E_convergence_H1} Grant Assumption \aswithargsref{assu:sigmaB_local}{S_{loc}}{0}{0}.
If $f\in H{}^{1}(\R^{d})$, then
\begin{equation*}
\Delta_{n}^{-1}E_{t,n}\left(f\right)\stackrel{ucp}{\to}\frac{1}{2}(f(Y_{t})-f(Y_{0}))-\frac{1}{2}\int_{0}^{t}\langle\nabla f(Y_{r}),\sigma_{r}dW_{r}\rangle.
\end{equation*}
\end{prop}

Since the limit process in Proposition \ref{prop:M_convergence_H1} is continuous, stable convergence also holds
at any fixed $0\leq t\leq T$, cf. \citep{billingsley2013convergence}.  It follows from Proposition \ref{prop:E_convergence_H1} using Slutsky's lemma that $\Delta_n^{-1}(M_{t,n}(f)+E_{t,n}(f))$ converges stably to the claimed limit in Theorem \ref{thm:CLT_Hs}.  The proof of the theorem follows therefore from showing that $\Delta_{n}^{-1}D_{t,n}(f)$ vanishes in probability asymptotically as $n\rightarrow\infty$, which we will do in Section \ref{sec:D} below.

We end this section with the proofs of the two aforementioned propositions.  It is worth emphasising that they hold for $f\in H^1(\R^d)$, while the analysis for $D_{t,n}(f)$ requires more smoothness for $f$. The crucial steps in the proof of Proposition \ref{prop:M_convergence_H1} are to suitably approximate the summands of $M_{t,n}(f)$ and to conclude by the (stochastic) Fubini theorem in \eqref{eq:Z_tilde}.

\begin{proof}[Proof of Proposition \ref{prop:M_convergence_H1}]
Recall $\widetilde{X}_{r}(t_{k-1})$ for $t_{k-1}\leq r\leq t_{k}$
from \eqref{eq:X_approx}.  Let
\begin{align*}
Z_{k} & =\int_{t_{k-1}}^{t_{k}}\left(f(Y_{r})-f(Y_{t_{k-1}})-\E\left[\l{f(Y_{r})-f(Y_{t_{k-1}})}\c F_{t_{k-1}}\right]\right)dr,\\
\widetilde{Z}_{k} & =\int_{t_{k-1}}^{t_{k}}\langle\nabla f(Y_{t_{k-1}}),\widetilde{X}_{r}(t_{k-1})-X_{t_{k-1}}-\E\left[\l{\widetilde{X}_{r}(t_{k-1})-X_{t_{k-1}}}\c F_{t_{k-1}}\right]\rangle dr,
\end{align*}
and write $M_{t,n}(f)=\sum_{k=1}^{\lfloor t/\Delta_{n}\rfloor}Z_{k}$,
$\widetilde{M}_{t,n}(f)=\sum_{k=1}^{\lfloor t/\Delta_{n}\rfloor}\widetilde{Z}_{k}$
and set $M_{t,n}=M_{t,n}(f)-\widetilde{M}_{t,n}(f)$.  $(M_{k\Delta_{n},n})_{k\in\{0,\dots,n\}}$
is a discrete martingale such that by the Burkholder-Gundy inequality
\begin{align*}
\E\bigg[\sup_{0\leq t\leq T}M_{t,n}^{2}\bigg] & =\E\bigg[\sup_{k\in\{1,\dots,n\}}M_{k\Delta_{n},n}^{2}\bigg]\leq 2\sum_{k=1}^{n}\E\left[\left(Z_{k}-\widetilde{Z}_{k}\right)^{2}\right].
\end{align*}
In addition, set
\[
\check{Z}_{k}=\int_{t_{k-1}}^{t_{k}}\langle\nabla f(Y_{t_{k-1}}),X_{r}-X_{t_{k-1}}-\E\left[\l{X_{r}-X_{t_{k-1}}}\c F_{t_{k-1}}\right]\rangle dr.
\]
Lemma \ref{lem:estimates_f}(ii,iii) shows $\sum_{k=1}^{n}\E[(Z_{k}-\check{Z}_{k})^{2}]=o(\Delta_{n}^{2})$
and $\sum_{k=1}^{n}\E[(\widetilde{Z}_{k}-\check{Z}_{k})^{2}]=o(\Delta_{n}^{3})$, hence we can conclude 
\begin{equation*}
\D^{-1}\sup_{0\leq t\leq T}\left|M_{t,n}\right|\r{\P}0.
\end{equation*}
It is therefore enough to study the limit of $\Delta_n^{-1}\widetilde{M}_{t,n}(f)$. The
claim follows from Theorem IX.7.28 of \citep{jacod2013limit}, once
we have shown for $0\leq t\leq T$ that
\begin{align}
\Delta_{n}^{-2}\sum_{k=1}^{\lfloor t/\Delta_{n}\rfloor}\E\Big[\t{\widetilde{Z}_{k}^{2}}\c F_{t_{k-1}}\Big]\r{\P} & \frac{1}{3}\int_{0}^{t}\left|\sigma_{r}^{\top}\nabla f(Y_{r})\right|{}^{2}dr,\label{eq:J1}\\
\Delta_{n}^{-2}\sum_{k=1}^{\lfloor t/\Delta_{n}\rfloor}\E\Big[\t{\widetilde{Z}_{k}^{2}\I_{\left\{ \left|\widetilde{Z}_{k}\right|>\epsilon\right\} }}\c F_{t_{k-1}}\Big]\r{\P} & 0,\,\,\,\,\text{for all }\epsilon>0,\label{eq:J2}\\
\Delta_{n}^{-1}\sum_{k=1}^{\lfloor t/\Delta_{n}\rfloor}\E\left[\l{\widetilde{Z}_{k}\left(W_{t_{k}}-W_{t_{k-1}}\right)^{\top}}\c F_{t_{k-1}}\right]\r{\P} & \frac{1}{2}\int_{0}^{t}\nabla f(Y_{r})^{\top}\sigma_{r}dr,\label{eq:J3}\\
\Delta_{n}^{-1}\sum_{k=1}^{\lfloor t/\Delta_{n}\rfloor}\E\left[\l{\widetilde{Z}_{k}\left(N_{t_{k}}-N_{t_{k-1}}\right)}\c F_{t_{k-1}}\right]\r{\P} & 0,\label{eq:J4}
\end{align}
where (\ref{eq:J4}) has to hold for all bounded ($\R$-valued) martingales
$N$ which are orthogonal to all components of $W$. Note that $\E[\widetilde{Z}_k|\mathcal{F}_k]=0$ such that the asymptotic bias $\Delta_n^{-1} \sum_{k=1}^{\lfloor t/\Delta_{n}\rfloor} \E[\widetilde{Z}_k|\mathcal{F}_{t_{k-1}}]$ vanishes. 

Let us prove (\ref{eq:J1}) through (\ref{eq:J4}).  Write $\tilde{X}_{r}(t_{k-1})-X_{t_{k-1}}=X_{t_{k-1}}+b_{t_{k-1}}\int_{t_{k-1}}^{r}dr' + \sigma_{t_{k-1}}\int_{t_{k-1}}^r dW_{r'}$.  The stochastic
Fubini theorem thus provides the identity
\begin{align}
\widetilde{Z}_{k}
	& =\int_{t_{k-1}}^{t_{k}}\langle\nabla f(Y_{t_{k-1}}),\sigma_{t_{k-1}}\int_{t_{k-1}}^r dW_{r'}-\sigma_{t_{k-1}}\E\left[\l{\int_{t_{k-1}}^r dW_{r'}}\c F_{t_{k-1}}\right]\rangle dr\nonumber\\
	& =\langle\nabla f(Y_{t_{k-1}}),\sigma_{t_{k-1}}\int_{t_{k-1}}^{t_{k}}(t_{k}-r)dW_{r}\rangle.\label{eq:Z_tilde}
\end{align}
By Itô's isometry, the boundedness of $\sigma$ from Assumption \aswithargsref{assu:sigmaB_local}{S_{loc}}{0}{0} and Lemma \ref{lem:estimates_f},
(\ref{eq:J1}) follows from
\begin{align*}
\Delta_{n}^{-2}\sum_{k=1}^{\lfloor t/\Delta_{n}\rfloor}\E\Big[\t{\widetilde{Z}_{k}^{2}}\c F_{t_{k-1}}\Big] & =\frac{\Delta_{n}}{3}\sum_{k=1}^{\lfloor t/\Delta_{n}\rfloor}\left|\sigma_{t_{k-1}}^{\top}\nabla f(Y_{t_{k-1}})\right|^{2}+o_{\P}\left(1\right)\\
 & =\frac{1}{3}\int_{0}^{t}\left|\sigma_{r}^{\top}\nabla f(Y_{r})\right|^{2}dr+o_{\P}\left(1\right),
\end{align*}
using Lemma \ref{lem:estimates_f}(iv) for the Riemann approximation
in the last line. With respect to \eqref{eq:J2} apply the Cauchy-Schwarz
inequality to $\tilde{Z}_k$ such that using the boundedness of $\sigma$ 
\begin{align*}
	|\widetilde{Z}_{k}| \lesssim h(\xi), \quad h(x) =  |\nabla f(x)|\Delta_{n}^{3/2}|V_k|
\end{align*}
and a random variable $V_k\overset{d}{\sim}N(0,I_{d})$, which is independent of $\c F_{t_{k-1}}$.  Since also $\xi$ is independent of $\c F_{t_{k-1}}$,  the first inequality of Lemma \ref{lem:plancherel} applied to $h$ yields for some $\epsilon'>0$
\begin{align*}
\E[\widetilde{Z}_{k}^{2}\I_{\{|\widetilde{Z}_{k}|>\epsilon\}}|\c F_{t_{k-1}}] & \lesssim \E[h^2(X_{t_{k-1}}+\xi)\I_{\{h(X_{t_{k-1}}+\xi) > \epsilon' \}}|\c F_{t_{k-1}}]\\
 & \lesssim\Delta_{n}^{3}\int_{\R^{d}}\E[|\nabla f(X_{t_{k-1}}+ x)|^{2}|V_k|^{2} \I_{\{h(X_{t_{k-1}}+x) > \epsilon' \}}|\c F_{t_{k-1}}] dx\\
 & = \Delta_{n}^{3}\int_{\R^{d}}|\nabla f(x)|^{2}\E[|V|^{2} \I_{\{h(x) > \epsilon' \}}|\c F_{t_{k-1}}] dx.
\end{align*}
For fixed $x\in\R^d$ we have $h(x)\rightarrow 0$ $\P$-a.s., implying $\E[|V_k|^{2} \I_{\{h(x) > \epsilon' \}}|\c F_{t_{k-1}}]\rightarrow 0$ by dominated convergence.  Hence,  \eqref{eq:J2} is obtained from using dominated convergence once more in the last display.  (\ref{eq:J3}) follows from
Itô's isometry and Lemma \ref{lem:estimates_f}(iv):
\begin{align*}
\Delta_{n}^{-1}\sum_{k=1}^{\lfloor t/\Delta_{n}\rfloor}\E\left[\l{\widetilde{Z}_{k}\left(W_{t_{k}}-W_{t_{k-1}}\right)^{\top}}\c F_{t_{k-1}}\right] & =\frac{\Delta_{n}}{2}\sum_{k=1}^{\lfloor t/\Delta_{n}\rfloor}\nabla f(Y_{t_{k-1}})^{\top}\sigma_{t_{k-1}}+o_{\P}(1)\\
 & =\frac{1}{2}\int_{0}^{t}\nabla f(Y_{r})^{\top}\sigma_{r}dr+o_{\P}(1).
\end{align*}
In the same way, (\ref{eq:J4}) follows from $\E[\widetilde{Z}_{k}(N_{t_{k}}-N_{t_{k-1}})|\c F_{t_{k-1}}]=o_{\P}(1)$. 
\end{proof}

\begin{proof}[Proof of Proposition \ref{prop:E_convergence_H1}]
Let $B_{k}=\langle\nabla f(Y_{t_{k-1}}),X_{t_{k}}-X_{t_{k-1}}\rangle$,
$A_{k}=f(Y_{t_{k}})-f(Y_{t_{k-1}})-B_{k}$, and write $E_{t,n}(f)=S_{1}(t)+S_{2}(t)+S_{3}(t)$
with
\begin{align*}
S_{1}(t) & =\frac{\Delta_{n}}{2}\sum_{k=1}^{\lfloor t/\Delta_{n}\rfloor}\left(f(Y_{t_{k}})-f(Y_{t_{k-1}})\right)=\frac{\Delta_{n}}{2}(f(Y_{\lfloor t/\Delta_{n}\rfloor\Delta_{n}})-f(Y_{0})),\\
S_{2}(t) & =\frac{\Delta_{n}}{2}\sum_{k=1}^{\lfloor t/\Delta_{n}\rfloor}\left(\E\left[\l{A_{k}}\c F_{t_{k-1}}\right]-A_{k}\right),\\
S_{3}(t) & =\frac{\Delta_{n}}{2}\sum_{k=1}^{\lfloor t/\Delta_{n}\rfloor}\left(\E\left[\l{B_{k}}\c F_{t_{k-1}}\right]-B_{k}\right).
\end{align*}
Lemma \ref{lem:estimates_f}(v) implies $\Delta_{n}^{-1}S_{1}(t)\r{ucp}\frac{1}{2}(f(Y_{t})-f(Y_{0}))$,
while we have by the Burkholder-Gundy inequality
\begin{equation}
\E\left[\sup_{0\leq t\leq T} S_{2}^{2}(t)\right]=\E\left[\sup_{k\in \{1,\dots,n\}} S_{2}^{2}(t)\right]\lesssim\Delta_{n}^{2}\sum_{k=1}^{n}\E[A_{k}^{2}]=o(\Delta_{n}^{2}),\label{eq:E_2_2}
\end{equation}
such that $\Delta_n^{-1}S_{2}(t)\r{ucp}0$, concluding with Lemma \ref{lem:estimates_f}(iii) (taking $Z_{r}=X_{\lceil r/\Delta_{n}\rceil\Delta_{n}}-X_{\lfloor r/\Delta_{n}\rfloor\Delta_{n}}$).  At last, decompose 
\begin{align*}
S_{3}(t) & =\frac{\Delta_{n}}{2}\sum_{k=1}^{\lfloor t/\Delta_{n}\rfloor}\int_{t_{k-1}}^{t_{k}}\sc{\nabla f(Y_{t_{k-1}})}{(\E[b_{r}|\c F_{t_{k-1}}]-b_{r})}dr\\
 & \quad-\frac{\Delta_{n}}{2}\sum_{k=1}^{\lfloor t/\Delta_{n}\rfloor}\int_{t_{k-1}}^{t_{k}}\sc{\nabla f(Y_{t_{k-1}})}{\sigma_{r}dW_{r}}.
\end{align*}
Exactly as in (\ref{eq:E_2_2}), but using Lemma \ref{lem:estimates_f}(i),
the first line is of order $o_{\P}(\Delta_{n})$ uniformly in $t$,  while the second
one equals $-\frac{\Delta_{n}}{2}\int_{0}^{t}\langle\nabla f(Y_{r}),\sigma_{r}dW_{r}\rangle+o_{\P}(\Delta_{n})$, again uniformly in $t$,
by Lemma \ref{lem:estimates_f}(iv) and Itô's isometry. 
\end{proof}

\subsubsection{The term $D_{t,n}(f)$}\label{sec:D}

Since $f$ is not smooth, Itô's formula cannot be used directly to reduce $D_{t,n}(f)$ to more manageable terms.  Instead,  write $D_{t,n}(f) = h(\xi)$ with
\begin{align*}
	h(x) &= \sum_{k=1}^{\lfloor t/\Delta_{n}\rfloor}\int_{t_{k-1}}^{t_{k}}\E[f(X_{r} + x)- f(X_{t_{k-1}} + x)\\
	&\quad\quad\quad-\frac{f(X_{t_k} + x)-f(X_{t_{k-1}}+x)}{2}|\c F_{t_{k-1}}]dr
\end{align*}
such that by  linearity of the Fourier transform 
\begin{align*}
	\mathcal{F}h(u) &=  \mathcal{F}f(u) \sum_{k=1}^{\lfloor t/\Delta_{n}\rfloor}\int_{t_{k-1}}^{t_{k}}\E[e^{-i\langle u,X_r\rangle} - e^{-i\langle u,X_{t_{k-1}}\rangle} \\
	&\quad\quad -\frac{e^{-i\langle u,X_{t_k}\rangle}-e^{-i\langle u,X_{t_{k-1}}}\rangle}{2}|\c F_{t_{k-1}}]dr
\end{align*}
for $u\in\R^d$.  For fixed $u$ the function $e^{i\sc u{\cdot}}$ is smooth and so we deduce from Itô's formula and Fubini's theorem that $\mathcal{F}h(u)=\mathcal{F}f(u)(F_{t,n}^{(1)}(u)+F_{t,n}^{(2)}(u))$
with
\begin{align}
F_{t,n}^{(1)}(u)= & -\sum_{k=1}^{\lfloor t/\Delta_{n}\rfloor}\int_{t_{k-1}}^{t_{k}}\left(t_{k}-r-\frac{\D}{2}\right)\E\bigg[ie^{-i\sc u{X_{r}}}\left\langle u,b_{r}\right\rangle \bigg|\c F_{t_{k-1}}\bigg]dr,\label{eq:F_t_n_1}\\
F_{t,n}^{(2)}(u)= & -\frac{1}{2}\sum_{k=1}^{\lfloor t/\Delta_{n}\rfloor}\int_{t_{k-1}}^{t_{k}}\left(t_{k}-r-\frac{\D}{2}\right)\E\bigg[e^{-i\sc u{X_{r}}}\left|\sigma_{r}^{\top}u\right|^{2}\bigg|\c F_{t_{k-1}}\bigg]dr.\label{eq:F_t_n_2}
\end{align}
Lemma \ref{lem:plancherel} therefore implies
\begin{align}
	\E[|D_{t,n}(f)|^2] \lesssim \int_{\R^d} |\c Ff(u)|^2 \E\left[|F_{t,n}^{(1)}(u)+F_{t,n}^{(2)}(u)|^2\right] du.\label{eq:D}
\end{align}
Introduce for $u\in\R^d$ and $i\in\{1,2\}$ the functions
\begin{align}
	g_n^{(i)}(u) = \Delta_n^{-2}(1+|u|^2)^{-s}\E\left[\sup_{0\leq t\leq T}|F_{t,n}^{(i)}(u)|^{2}\right].\label{eq:gFun}
\end{align}
This provides us with 
\begin{align*}
	\E[|D_{t,n}(f)|^2] \lesssim \Delta_n^2 \int_{\R^d} |\c Ff(u)|^2 (1+|u|^2)^s (g_n^{(1)}(u)+g_n^{(2)}(u)) du.
\end{align*}
Lemmas \ref{lem:F_1_bound} and \ref{lem:F_2_bound_SM} below together with dominated convergence therefore yield the following result, which concludes the proof of Theorem \ref{thm:CLT_Hs}.

\begin{prop}
\label{prop:Drift_Convergence_Hs} Let $s\geq 1$ and grant Assumption \aswithargsref{assu:sigmaB_local}{S_{loc}}{\alpha}{\beta} with 
$\alpha>\max(0,1-s/2)$, $\beta>0$. Then we have for $f\in H^{s}(\R^{d})$
and $0\leq t\leq T$
\begin{equation}
\Delta_{n}^{-1}D_{t,n}\left(f\right)\r{\P}0.
\end{equation}
\end{prop}

Let us now state and prove the aforementioned two lemmas, as well as two auxiliary lemmas.

\begin{lem}
\label{lem:F_1_bound}Under Assumption \aswithargsref{assu:sigmaB_local}{S_{loc}}{\alpha}{\beta} with $\beta>0$ the function $g_n^{(1)}$ from \eqref{eq:gFun} with $s\geq1$  satisfies $g_n^{(1)}(u)\rightarrow 0$ as $n\rightarrow\infty$ for all $u\in\R^d$ and $\sup_{n\in\N, u\in\R^d}g_n^{(1)}(u) <\infty$.
\end{lem}

\begin{proof}
Define $\tilde{F}_{t,n}^{(1)}(u)$ exactly as $F_{t,n}^{(1)}(u)$, but with
$e^{-i\sc u{X_{r}}}\sc u{b_{r}}$ replaced by $e^{-i\sc u{X_{\lfloor r/\Delta_{n}\rfloor\Delta_{n}}}}\sc u{b_{\lfloor r/\Delta_{n}\rfloor\Delta_{n}}}$.  From $\int_{t_{k-1}}^{t_{k}}(t_{k}-r-\D/2)dr=0$ conclude  
$\tilde{F}_{t,n}^{(1)}(u)=0$ such that
\begin{align*}
	g_n^{(1)}(u) 
		&= \Delta_n^{-2}(1+|u|^2)^{-s}\E\left[\sup_{0\leq t\leq T}|F_{t,n}^{(1)}(u)-\tilde{F}_{t,n}^{(1)}(u)|^{2}\right]\\
		&\lesssim \E\left[\sup_{0\leq t\leq T}|e^{-i\sc u{X_{t}}}\sc{u/|u|}{b_{t}}-e^{-i\sc u{X_{\lfloor t/\Delta_{n}\rfloor\Delta_{n}}}}\sc{u/|u|}{b_{\lfloor t/\Delta_{n}\rfloor\Delta_{n}}}|^{2}\right]
\end{align*}
using in the last line $s\geq 1$.  From Assumption \aswithargsref{assu:sigmaB_local}{S_{loc}}{0}{0}, the process $b$ is uniformly bounded such that $\sup_{n\in\N, u\in\R^d}g_n^{(1)}(u) <\infty$. On the other hand,  as $X$ has càdlàg paths and using the approximation property of $b$ according to \aswithargsref{assu:sigmaB_local}{S_{loc}}{0}{0} for $\beta>0$, conclude also that $g_n^{(1)}(u)\rightarrow 0$ as $n\rightarrow\infty$ for $u\in\R^d$. 
\end{proof}

The next lemma is the key result for general continuous semimartingales and is inspired by the one-step-Euler-approximation of \citep{Fournier:2010ed} to approximate the characteristic function of the marginals $X_t$.
\begin{lem}
\label{lem:F_2_bound_SM} Let $s\geq 1$ and grant Assumption \aswithargsref{assu:sigmaB_local}{S_{loc}}{\alpha}{\beta} for $\alpha>\max(0,1-s/2)$, $\beta>0$. Then the function $g_n^{(2)}$ from \eqref{eq:gFun} satisfies $g_n^{(2)}(u)\rightarrow 0$ as $n\rightarrow\infty$ for all $u\in\R^d$ and $\sup_{n\in\N, u\in\R^d}g_n^{(2)}(u) <\infty$. 
\end{lem}

\begin{proof}
We distinguish the cases $|u|\leq 1$ and $|u|>1$. For $|u|\leq 1$ the argument is analogous to the proof of Lemma \ref{lem:F_1_bound}.  Define $\tilde{F}_{t,n}^{(2)}(u)$ exactly as $F_{t,n}^{(2)}(u)$, but with
$e^{-i\sc u{X_{r}}}|\sigma_{r}^{\top}u|^2$ replaced by $e^{-i\sc u{X_{\lfloor r/\Delta_{n}\rfloor\Delta_{n}}}}|\sigma_{\lfloor r/\Delta_{n}\rfloor\Delta_{n}}^{\top}u|^2$. Again $\tilde{F}_{t,n}^{(2)}(u)=0$ and we conclude as above that 
\begin{align*}
	g_n^{(2)}(u) 
		&= \Delta_n^{-2}(1+|u|^2)^{-s}\E\left[\sup_{0\leq t\leq T}|F_{t,n}^{(2)}(u)-\tilde{F}_{t,n}^{(2)}(u)|^{2}\right]\\
		&\lesssim \E\left[\sup_{0\leq t\leq T}|e^{-i\sc u{X_{t}}}|\sigma_{t}^{\top}u|^2-e^{-i\sc u{X_{\lfloor t/\Delta_{n}\rfloor\Delta_{n}}}}|\sigma_{\lfloor t/\Delta_{n}\rfloor\Delta_{n}}^{\top}u|^2|^{2}\right]
\end{align*}
satisfies $\sup_{n\in\N, u\in\R^d,|u|\leq 1}g_n^{(2)}(u) <\infty$ and $g_n^{(2)}(u)\rightarrow 0$ as $n\rightarrow\infty$ for $|u|\leq 1$.

Let now $|u|>1$.  We introduce a new grid depending on the
parameters 
\begin{equation}
\epsilon\equiv\epsilon(u,\Delta_{n})=\nu |u|^{-2},\,\,\,\nu=\nu(u,\Delta_n)=C_{1}^{-1}\log(1+|u|^{6}\eta_n^{1/2}),\label{eq:eps}
\end{equation}
where $C_{1}>0$ is a constant such that according to  Assumption \aswithargsref{assu:sigmaB_local}{S_{loc}}{\alpha}{\beta} $\inf_{r}(\frac{1}{2}|\sigma_{r}^{\top}u|^{2})=\inf_{r}(\frac{1}{2}\langle \sigma_r\sigma_r^\top u,u\rangle)\geq C_{1}|u|^{2}$ and where $\eta_n$ is a sequence of non-negative real numbers to be determined later. 
Recall the approximated process from \eqref{eq:X_approx} and set for $0\leq r,r',r''\leq T$ 
\begin{align*}
	U_{r,r',r''} & =\E[e^{-i\sc u{\widetilde{X}_{r}(r')}}|\sigma_{r''}^{\top}u|^{2}|\c F_{t_{k-1}}],\,\,U_{r,r'}=U_{r,r',r'}.
\end{align*}
With this define
\begin{align*}
\check{F}_{t,n}^{(2)}(u) & =\sum_{k=1}^{\lfloor t/\Delta_{n}\rfloor}\int_{t_{k-1}}^{t_{k}}(t_{k}-r-\frac{\D}{2})U_{r,r(\epsilon),r(\epsilon)}dr.
\end{align*}
and obtain the upper bound $g_n^{(2)}(u) \leq 2 g_n^{(2,1)}(u) + 2g_n^{(2,2)}(u)$
with 
\begin{align*}
	g_n^{(2,1)}(u) 
		&= \Delta_n^{-2}(1+|u|^2)^{-s}\E\left[\sup_{0\leq t\leq T}|F_{t,n}^{(2)}(u)-\check{F}_{t,n}^{(2)}(u)|^{2}\right],\\
	g_n^{(2,2)}(u) 
		&=  \Delta_n^{-2}(1+|u|^2)^{-s}\E\left[\sup_{0\leq t\leq T}|\check{F}_{t,n}^{(2)}(u)|^{2}\right].
\end{align*}
Upper bounds on these two terms are obtained in Lemmas \ref{lem:F_2_SM_err_1}, \ref{lem:F_2_SM_err_2} and \ref{lem:F_2_SM_3} below.  The null sequence $\eta_n$ is determined in Lemma \ref{lem:F_2_SM_err_2}. In order to conclude when $|u|>1$ note that the conditions $\alpha>\max(0,1-s/2)$, $\beta>0$, $s\geq1$  imply for some sufficiently small $\delta>0$
\begin{align}
	& |u|^{4-2s}\epsilon^{2\alpha} +|u|^{6-2s}\epsilon^{2+2\beta}+|u|^{6-2s}\epsilon^{1+2\alpha}\leq \epsilon^{\delta}\nu^{2\alpha} +|u|^{4}\epsilon^{2+2\beta}+|u|^2 \epsilon^{1+\delta}\nu^{2\alpha}\nonumber\\
	&\quad\quad \lesssim \epsilon^{\delta}\nu^{2\alpha} + \epsilon^{\delta}\nu^{1+2\alpha},\label{eq:g1}\\
	&|u|^{4-2s}\epsilon e^{-C_1|u|^{2}\epsilon}+|u|^{2-2s} (1-e^{-C_1|u|^{2}\Delta_{n}})+|u|^{8-2s}e^{-C_{1}|u|^{2}\epsilon}\eta_{n}\nonumber\\
	&\quad\quad \leq |u|^{2}\epsilon e^{-C_1|u|^{2}\epsilon}+(1-e^{-C_1|u|^{2}\Delta_{n}})+|u|^{6}e^{-C_{1}|u|^{2}\epsilon}\eta_{n}\nonumber\\
	&\quad\quad= \nu e^{-C_1\nu}+(1-e^{-C_1|u|^{2}\Delta_{n}})+\frac{|u|^{6}\eta_{n}}{1+|u|^6\eta_n},\label{eq:g2}\\
	&|u|^{4-2s}\epsilon^{2}+|u|^{2-2s}(1-e^{-C_1|u|^{2}\epsilon})+|u|^{8-2s}e^{-C_{1}|u|^{2}\epsilon}\eta_{n}\nonumber\\
	&\quad\quad \leq |u|^{2}\epsilon^{2}+(1-e^{-C_1|u|^{2}\epsilon})+|u|^{6}e^{-C_{1}|u|^{2}\epsilon}\eta_{n}\nonumber\\
	&\quad\quad \leq \epsilon \nu+(1-e^{-C_1\nu})+\frac{|u|^{6}\eta_{n}}{1+|u|^6\eta_n}.\label{eq:g3}
\end{align}
We get from \eqref{eq:eps} that $\epsilon\leq 1$ and $\nu^p\lesssim \epsilon$ uniformly in $u,n$ and any $p\in\N$ and that $\epsilon,\nu \rightarrow 0$ for any fixed $u\in\R^d$ as $n\rightarrow\infty$.  
Consequently, the terms in \eqref{eq:g1}, \eqref{eq:g2} and \eqref{eq:g3} are uniformly in $u,n$ bounded and converge to zero for any fixed $u\in\R^d$ as $n\rightarrow \infty$.  Lemmas \ref{lem:F_2_SM_err_1}, \ref{lem:F_2_SM_err_2} and \ref{lem:F_2_SM_3} then show that $g_n^{(2,1)}(u)$ and $g_n^{(2,2)}(u)$ and thus also $g_n^{(2)}(u)$ are uniformly in $u,n$ bounded and converge to zero for any fixed $u\in\R^d$ as $n\rightarrow\infty$, which is what we wanted to prove.
\end{proof}

\begin{lem}
\label{lem:F_2_SM_err_1} In Lemma \ref{lem:F_2_bound_SM} it holds $$g_n^{(2,1)}(u)\lesssim |u|^{4-2s}\epsilon^{2\alpha} +|u|^{6-2s}\epsilon^{2+2\beta}+|u|^{6-2s}\epsilon^{1+2\alpha}.$$
\end{lem}

\begin{proof}
Let $k\geq1$, $t_{k-1}\leq r<t_{k}$. Since $r-r(\epsilon)\lesssim\epsilon$, Assumption \aswithargsref{assu:sigmaB_local}{S_{loc}}{\alpha}{\beta} and  Lemma \ref{lem:estimatesForSM}(ii) provide the approximation errors 
\begin{align*}
	&\E[\sup_{0\leq r\leq T}|\sigma_{r}-\sigma_{r(\epsilon)}|{}^{2}]\lesssim\epsilon^{2\alpha},\quad \E[|X_{r}-\widetilde{X}_{r}(r(\epsilon))|^{2}]\lesssim\epsilon^{2\beta+2}+\epsilon^{2\alpha+1}.
\end{align*}
Since $\tilde{X}_r(r)=X_r$ and recalling that $\sigma$ is uniformly bounded by Assumption \aswithargsref{assu:sigmaB_local}{S_{loc}}{\alpha}{\beta},  this further gives
\begin{align*}
	\E[|U_{r,r}-U_{r,r(\epsilon)}|^{2}] 
		& \lesssim  \E\left[\left| |\sigma_{r}^{\top}u|^{2} - |\sigma_{r(\epsilon)}^{\top}u|^{2}\right|^2\right] + |u|^4 \E\left[\left|e^{-i\sc u{X_{r}}}-e^{-i\sc u{\widetilde{X}_{r}(r(\epsilon))}}\right|^{2}\right]\\
		& \lesssim |u|^{4}\epsilon^{2\alpha} +|u|^{6}\epsilon^{2+2\beta}+|u|^{6}\epsilon^{1+2\alpha}.
\end{align*}
The result follows thus from
\begin{align*}
	(g_n^{(2,1)}(u))^{1/2} 
		& = \Delta_n^{-1}|u|^{-s}\E\left[\sup_{0\leq t\leq T}|F_{t,n}^{(2)}(u)-\check{F}_{t,n}^{(2)}(u)|^{2}\right]^{1/2}\\
		& \leq \Delta_n^{1/2} |u|^{-s} \sum_{k=1}^n \left(\int_{t_{k-1}}^{t_k} \E[|U_{r,r}-U_{r,r(\epsilon)}|^{2}] dr \right)^{1/2}\\
		& \lesssim  |u|^{2-s}\epsilon^{\alpha} +|u|^{3-s}\epsilon^{1+\beta}+|u|^{3-s}\epsilon^{1/2+\alpha}.
\end{align*}
\end{proof}
\begin{lem}
\label{lem:F_2_SM_err_2} In the setting of Lemma \ref{lem:F_2_bound_SM} there exists a sequence $\eta_n\geq 0$ with $\eta_n\rightarrow 0$ and such that for $\epsilon \leq \Delta_n$  we have $$g_n^{(2,2)}(u)\lesssim |u|^{4-2s}\epsilon e^{-C_1|u|^{2}\epsilon}+|u|^{2-2s} (1-e^{-C_1|u|^{2}\Delta_{n}})+|u|^{8-2s}e^{-C_{1}|u|^{2}\epsilon}\eta_{n}.$$ 
\end{lem}

\begin{proof}
Observe first the following fact by the Burkholder-Gundy
inequality: If $(\mathcal{G}_{k})_{k\in\{1,\dots,K\}}$ for $K\in\N$
is a discrete filtration with $\c G_{k}$-measurable and square
integrable random variables $R_{k}$, then
\begin{align}
\E[\max_{m\in\{1,\dots,K\}}|\sum_{k=1}^{m}R_{k}|^{2}] & \lesssim\sum_{k=1}^{K}\E[|R_{k}|^{2}]+\E[\sup_{m\in\{1,\dots,K\}}|\sum_{k=1}^{m}\E[R_{k}|\c G_{k-1}]|^{2}].\label{eq:BDG}
\end{align}
Let now $|u|>1$ and $\epsilon\leq\Delta_{n}$. Set $Z_{k}:=\int_{t_{k-1}}^{t_{k}}(t_{k}-r-\frac{\D}{2})U_{r,r(\epsilon)}dr$
for $k\in\{1,\dots,n\}$ such that
\[
\E[\sup_{0\leq t\leq T}|\check{F}_{t,n}^{(2)}(u)|^{2}]\lesssim\E[|Z_{1}|^{2}]+\E[\max_{m\in\{2,\dots,n\}}|\sum_{k=2}^{m}Z_{k}|^{2}].
\]
We will show that there exist $0\leq\eta_{n}\rightarrow0$ as $n\rightarrow\infty$ such that
\begin{align}
& k\geq1: \,\E[|Z_{k}|^{2}]\lesssim\Delta_{n}^{3}|u|^{2}\left(|u|^{2}\epsilon e^{-C_1|u|^{2}\epsilon}+(1-e^{-C_1|u|^{2}\Delta_{n}})\right),\label{eq:Z_bound_1}\\
 & \,\E[\max_{k\in\{2,\dots,n\}}|\E[Z_{k}|\c F_{t_{k-2}}]|^{2}]\lesssim\Delta_{n}^{4}|u|^{8}e^{-C_{1}|u|^{2}\epsilon}\eta_{n}.\label{eq:Z_bound_2}
\end{align}
Assuming this holds,  we can apply \eqref{eq:BDG} to $R_{k}=Z_{k+1}$, $k\in\{1,\dots,n-1\}$,
$\c G_{k}=\c F_{t_{k}}$ such that by \eqref{eq:Z_bound_1} and \eqref{eq:Z_bound_2}
\begin{align*}
 & g_n^{(2,2)}(u) = \Delta_n^{-2}(1+|u|^2)^{-s} \E[\sup_{0\leq t\leq T}|\check{F}_{t,n}^{(2)}(u)|^{2}] \\
 & \quad \lesssim \Delta_n^{-2}|u|^{-2s}\left(\E[|Z_{1}|^{2}]+\Delta_{n}^{-1}\max_{k\in\{2,\dots,n\}}\E[|Z_{k}|^{2}]+\Delta_{n}^{-2}\max_{k\in\{2,\dots,n\}}\E[|\E[Z_{k}|\c F_{t_{k-2}}]|^{2}]\right)\\
 & \quad\lesssim |u|^{4-2s}\epsilon e^{-C_1|u|^{2}\epsilon}+|u|^{2-2s} (1-e^{-C_1|u|^{2}\Delta_{n}})+|u|^{8-2s}e^{-C_{1}|u|^{2}\epsilon}\eta_{n},
\end{align*}
which proves the claim.  

Let us next show (\ref{eq:Z_bound_1}) and (\ref{eq:Z_bound_2}).  Fix
$k\geq1$ and note that
\begin{align*}
\E[|Z_{k}|^{2}] & \lesssim\Delta_{n}^{2}\int_{t_{k-1}}^{t_{k}}\int_{t_{k-1}}^{t_{k}}|\E[U_{r,r(\epsilon)}\overline{U_{r',r'(\epsilon)}}]|drdr'.
\end{align*}
Let $t_{k-1}\leq r'\leq r<t_{k}$ and $r_{*}=\max(r(\epsilon),r')$.
From (\ref{eq:X_approx}) it follows that $\sc u{\widetilde{X}_{r}(r(\epsilon))-\widetilde{X}_{r_{*}}(r(\epsilon))}$
has conditional on $\c F_{r_{*}}$ is $N(\sc u{b_{r(\epsilon)}}(r-r_{*}),|\sigma_{r(\epsilon)}^{\top}u|^{2}(r-r_{*}))$-distributed.  Hence, 
\begin{align*}
	\left|\E[e^{-i\sc u{\widetilde{X}_{r}(r(\epsilon))-\widetilde{X}_{r_{*}}(r(\epsilon))}}|\c F_{r_{*}}]\right| = e^{\frac{1}{2}|\sigma_{r(\epsilon)}^{\top}u|^{2}(r-r_{*})} \leq e^{-C_1|u|^2(r-r_*)}
\end{align*}
with $C_1$ from \eqref{eq:eps}.  The tower property of conditional expectation and $|U_{r',r'(\epsilon)}|\lesssim |u|^2$ gives
\begin{align}
|\E[U_{r,r(\epsilon)}\overline{U_{r',r'(\epsilon)}}]| & =|\E[\E[e^{-i\sc u{\widetilde{X}_{r}(r(\epsilon))-\widetilde{X}_{r_{*}}(r(\epsilon))}}|\c F_{r_{*}}]e^{-i\sc u{\widetilde{X}_{r_{*}}(r(\epsilon))}}|\sigma_{r(\epsilon)}^{\top}u|^{2}\overline{U_{r',r'(\epsilon)}}]|\nonumber \\
 & \lesssim|u|^{4}\E[|\E[e^{-i\sc u{\widetilde{X}_{r}(r(\epsilon))-\widetilde{X}_{r_{*}}(r(\epsilon))}}|\c F_{r_{*}}]|]\lesssim|u|^{4}e^{-C_1|u|^{2}(r-r_{*})}.\label{eq:Ucross}
\end{align}
The condition $\epsilon\leq\Delta_{n}$ yields $\min(\epsilon,r-r')\leq r-r_{*}\leq\Delta_{n}$
such that
\begin{align*}
\E[|Z_{k}|^{2}] & \lesssim\Delta_{n}^{2}|u|^{4}\left(\Delta_{n}^{2}e^{-C_1|u|^{2}\epsilon}+\int_{t_{k-1}}^{t_{k}}\int_{t_{k-1}}^{t_{k}}e^{-C_1|u|^{2}(r-r')}drdr'\right)\\
 & \lesssim\Delta_{n}^{3}|u|^{2}\left(|u|^{2}\epsilon e^{-C_1|u|^{2}\epsilon}+(1-e^{-C_1|u|^{2}\Delta_{n}})\right).
\end{align*}
This proves (\ref{eq:Z_bound_1}). To see why (\ref{eq:Z_bound_2})
holds, let $k\geq2$. Since $\int_{t_{k-1}}^{t_{k}}(t_{k}-r-\Delta_{n}/2)dr$
vanishes, the same holds for $U_{t_{k},t_{k-2}}=U_{t_{k},t_{k-2},t_{k-2}}$
and thus
\begin{align*}
|\E[Z_{k}|\c F_{t_{k-2}}]| & \lesssim\Delta_{n}\int_{t_{k-1}}^{t_{k}}|\E[U_{r,r(\epsilon),r(\epsilon)}-U_{r,r(\epsilon),t_{k-2}}|\c F_{t_{k-2}}]|dr\\
 & \quad+\Delta_{n}\int_{t_{k-1}}^{t_{k}}|\E[U_{r,r(\epsilon),t_{k-2}}-U_{r,t_{k-2},t_{k-2}}|\c F_{t_{k-2}}]dr.
\end{align*}
Set $r^{*}=\max(r(\epsilon),t_{k-2})$ for $t_{k-1}\leq r<t_{k}$
and write
\begin{align*}
\widetilde{X}_{r}(r(\epsilon)) & =\sigma_{r(\epsilon)}(W_{r}-W_{r^{*}})+b_{r(\epsilon)}(r-r^{*})+\widetilde{X}_{r^{*}}(r(\epsilon)),\\
\widetilde{X}_{t_{k}}(t_{k-2}) & =\sigma_{t_{k-2}}(W_{t_{k}}-W_{r^{*}})+\sigma_{t_{k-2}}(W_{r}-W_{r^{*}})+b_{t_{k-2}}(t_{k}-r^{*})+\widetilde{X}_{r^{*}}(t_{k-2}).
\end{align*}
Conditioning on $\c F_{r^{*}}$ shows 
\begin{align*}
 & |\E[U_{r,r(\epsilon),r(\epsilon)}-U_{r,r(\epsilon),t_{k-2}}|\c F_{r^{*}}]|\\
 & \lesssim|u|^{2}|\sigma_{r(\epsilon)}-\sigma_{t_{k-2}}|\,|\E[e^{-i\sc u{\widetilde{X}_{r}(r(\epsilon))}}|\c F_{r^{*}}]|\lesssim|u|^{2}\Delta_{n}e^{-C|u|^{2}\epsilon},
\end{align*}
using that $2\Delta_{n}\geq r-r^{*}\geq\epsilon$ (because $\epsilon\leq\Delta_{n}$).
On the other hand, assuming first that $R=|\sigma_{r(\epsilon)}^{\top}u|^{2}-|\sigma_{t_{k-2}}^{\top}u|^{2}\geq0$,
we have
\begin{align*}
 & |\E[U_{r,r(\epsilon),t_{k-2}}-U_{t_{k},t_{k-2},t_{k-2}}|\c F_{r^{*}}]|\\
 & \lesssim|u|^{2}|\E[e^{-i\sc u{\widetilde{X}_{r}(r(\epsilon))}}-e^{-i\sc u{\widetilde{X}_{t_{k}}(t_{k-2})}}|\c F_{r^{*}}]|\\
 & =|u|^{2}e^{-\frac{1}{2}|\sigma_{t_{k-2}}^{\top}u|^{2}(r-r^{*})}|e^{-\frac{1}{2}R(r-r^{*})-i\sc u{b_{r(\epsilon)}(r-r^{*})+\widetilde{X}_{r^{*}}(r(\epsilon))}}\\
 & \qquad-e^{-\frac{1}{2}|\sigma_{t_{k-2}}^{\top}u|^{2}(t_{k}-r^{*})-i\sc u{b_{t_{k-2}}(r-r^{*})+\widetilde{X}_{r^{*}}(t_{k-2})}}|\\
 & \lesssim|u|^{4}e^{-C|u|^{2}\epsilon}(\Delta_{n}+|\widetilde{X}_{r^{*}}(r(\epsilon))-\widetilde{X}_{r^{*}}(t_{k-2})|),
\end{align*}
using in the last line $|u|>1$, $R\geq0$ and again $2\Delta_{n}\geq r-r^{*}\geq\epsilon$.
The same upper bound is obtained for $R<0$ by taking $e^{-\frac{1}{2}|\sigma_{r(\epsilon)}^{\top}u|^{2}(r-r^{*})}$
out of the absolute value above instead of $e^{-\frac{1}{2}|\sigma_{t_{k-2}}^{\top}u|^{2}(r-r^{*})}$.
We thus find
\begin{equation}
\E[\sup_{k\in\{1,\dots,n\},t_{k-1}\leq r<t_{k}}|\E[U_{r,r(\epsilon),t_{k-2}}-U_{t_{k},t_{k-2},t_{k-2}}|\c F_{r^{*}}]|^{2}]\lesssim|u|^{8}e^{-C_{1}|u|^{2}\epsilon}\eta_{n},\label{eq:eta_n}
\end{equation}
where $0\leq\eta_{n}\rightarrow0$ for $n\rightarrow\infty$ due to
Lemma \ref{lem:estimatesForSM}(iii). In all, this shows (\ref{eq:Z_bound_2}) and ends the proof.
\end{proof}
\begin{lem}
\label{lem:F_2_SM_3} In the setting of Lemma \ref{lem:F_2_bound_SM} we have for $\epsilon > \Delta_n$ that $$g_n^{(2,2)}(u)\lesssim |u|^{4-2s}\epsilon^{2}+|u|^{2-2s}(1-e^{-C_1|u|^{2}\epsilon})+|u|^{8-2s}e^{-C_{1}|u|^{2}\epsilon}\eta_{n}.$$ 
\end{lem}

\begin{proof}
Let $|u|>1$ and $\epsilon>\Delta_{n}$. We first fix some notation.
Let 
\[
I_{j}(t)=\{k=1,\dots,\lfloor t/\Delta_{n}\rfloor:(j-1)\varepsilon<t_{k}\leq j\varepsilon\},\,\,\,\,1\leq j\leq\lceil T/\epsilon\rceil,
\]
be the set of those $k\leq\lfloor t/\D\rfloor$ such that $t_{k}\leq t$
lies in the interval $((j-1)\varepsilon,j\varepsilon]$. Let $Z_{k}$
be as in Lemma \ref{lem:F_2_SM_err_2} and set $A_{t}^{(j)}=\sum_{k\in I_{j}(t)}Z_{k}$
such that $\check{F}_{t,n}^{(2)}(u)=\sum_{j=1}^{\lceil T/\epsilon\rceil}A_{t}^{(j)}$.
For $t\geq0$ denote by $j(t)$ the unique $j\in\{1,\dots,\lceil T/\epsilon\rceil\}$
with $(j-1)\varepsilon<t\leq j\varepsilon$. If $t\leq(j-1)\epsilon$,
then $I_{j}(t)$ is empty and $A_{t}^{(j)}=0$, while for $t>j\epsilon$
we have $I_{j}(t)=I_{j}(T)$ and $A_{t}^{(j)}=A_{T}^{(j)}$. This
means $\check{F}_{t,n}^{(2)}(u)=\sum_{j=1}^{j(t)-1}A_{T}^{(j)}+A_{t}^{(j(t))}$.
Using the trivial bound $|U_{r,r(\epsilon)}|\lesssim|u|^{2}$ for
$r\geq0$ and the fact that $I_{j}(t)$ contains at most $2\epsilon\Delta_{n}^{-1}$
many $k$, we get $|A_{t}^{(j)}|\lesssim\Delta_{n}|u|^{2}\epsilon$
for all $1\leq j\leq\lceil T/\epsilon\rceil$ and therefore
\[
\sup_{0\leq t\leq T}|\check{F}_{t,n}^{(2)}(u)|\lesssim\max_{m\in\{3,\dots,\lceil T/\epsilon\rceil\}}|\sum_{j=3}^{m}A_{T}^{(j)}|+\Delta_{n}|u|^{2}\epsilon.
\]
Applying (\ref{eq:BDG}) three times (first with $R_{k}=A_{T}^{(k+2)}\in\c G_{k}=\c F_{(k+2)\epsilon}$,
$k\in\{1,\dots,\lceil T/\epsilon\rceil-2\}$, then with $R_{k}=\E[A_{T}^{(k+2)}|\c F_{(k+1)\epsilon}]\in\c G_{k}=\c F_{(k+1)\epsilon}$,
and finally with $R_{k}=\E[A_{T}^{(k+2)}|\c F_{k\epsilon}]\in\c G_{k}=\c F_{k\epsilon}$) yields
\begin{align*}
  \E[\sup_{0\leq t\leq T}|\check{F}_{t,n}^{(2)}(u)|^{2}] 
 	& \lesssim \epsilon^{-1}\max_{j\in\{3,\dots,\lceil T/\epsilon\rceil\}}\E[|A_{T}^{(j)}|^{2}]\\
 	&\quad +\epsilon^{-2}\max_{j\in\{3,\dots,\lceil T/\epsilon\rceil\}}\E[|\E[A_{T}^{(j)}|\c F_{(j-3)\epsilon}]|^{2}]+\Delta_{n}^{2}|u|^{4}\epsilon^{2}.
\end{align*}
We show below for $j\geq3$ (cf. (\ref{eq:Z_bound_1}), (\ref{eq:Z_bound_2}))
that
\begin{align}
\E[|A_{T}^{(j)}|^{2}] & \lesssim\Delta_{n}^{2}|u|^{2}\epsilon(1-e^{-C_1|u|^{2}\epsilon}),\label{eq:A_t_1}\\
\E[|\E[A_{T}^{(j)}|\c F_{(j-3)\epsilon}]|^{2}] & \lesssim\Delta_{n}^{2}\epsilon^{2}|u|^{8}e^{-C_{1}|u|^{2}\epsilon}\eta_{n},\label{eq:A_t_2}
\end{align}
with $\eta_{n}$ from (\ref{eq:Z_bound_2}). Plugging
these bounds into the last display gives 
\begin{align*}
 & g_n^{(2,2)}(u) = \Delta_n^{-2}(1+|u|^2)^{-s} \E[\sup_{0\leq t\leq T}|\check{F}_{t,n}^{(2)}(u)|^{2}] \\
 & \quad \lesssim \Delta_n^{-2}|u|^{-2s}\left(\E[|Z_{1}|^{2}]+\Delta_{n}^{-1}\max_{k\in\{2,\dots,n\}}\E[|Z_{k}|^{2}]+\Delta_{n}^{-2}\max_{k\in\{2,\dots,n\}}\E[|\E[Z_{k}|\c F_{t_{k-2}}]|^{2}]\right)\\
 & \quad\lesssim |u|^{2-2s}(1-e^{-C_1|u|^{2}\epsilon})+|u|^{8-2s}e^{-C_{1}|u|^{2}\epsilon}\eta_{n}+|u|^{4-2s}\epsilon^{2}.
\end{align*}

Let us now prove (\ref{eq:A_t_1}) and (\ref{eq:A_t_2}). For (\ref{eq:A_t_1})
let $k,k'\in I_{j}(T)$, $j\geq3$ and consider $t_{k-1}\leq r<t_{k}$,
$t_{k'-1}\leq r'<t_{k'}$, $r'\leq r$. Since $\epsilon>\Delta_{n}$,
we have $r(\epsilon)\leq r'$, implying by \eqref{eq:Ucross} with
$r_{*}=r'$ that $|\E[U_{r,r(\epsilon)}\overline{U_{r',r'(\epsilon)}}]|\lesssim|u|^{4}e^{-C_1|u|^{2}(r-r')}$.
As also $(j-2)\epsilon\leq t_{k-1},t_{k'-1}$, this shows
\begin{align*}
\E[|A_{T}^{(j)}|^{2}] & \lesssim\Delta_{n}^{2}\sum_{k,k'\in I_{j}(T)}\int_{t_{k-1}}^{t_{k}}\int_{t_{k'-1}}^{t_{k'}}|\E[U_{r,r(\epsilon)}\overline{U_{r',r'(\epsilon)}}]|drdr'\\
 & \lesssim\Delta_{n}^{2}|u|^{4}\int_{(j-2)\epsilon}^{j\epsilon}\int_{(j-2)\epsilon}^{j\epsilon}e^{-C_1|u|^{2}\left|r-r'\right|}drdr'\\
 & \lesssim\Delta_{n}^{2}|u|^{2}\epsilon(1-e^{-C_1|u|^{2}\epsilon}),
\end{align*}
proving (\ref{eq:A_t_1}). For (\ref{eq:A_t_2}), on the other hand,
we have
\begin{align*}
& \E[|\E[A_{T}^{(j)}|\c F_{(j-3)\epsilon}]|^{2}] \lesssim\Delta_{n}^{2}\E[|\sum_{k\in I_{j}(T)}\int_{t_{k-1}}^{t_{k}}|\E[U_{r,r(\epsilon)}-U_{t_{k},(j-3)\epsilon}|\c F_{(j-3)\epsilon}]|dr|^{2}]\\
 & \quad \lesssim\Delta_{n}^{4}(\epsilon/\Delta_{n})^{2}\sup_{k\in\{1,\dots,n\},t_{k-1}\leq r<t_{k}}\E[|\E[U_{r,r(\epsilon)}-U_{t_{k},(j-3)\epsilon}|\c F_{(j-3)\epsilon}]|^{2}]\\
 & \quad  \lesssim\Delta_{n}^{2}\epsilon^{2}|u|^{8}e^{-C_{1}|u|^{2}\epsilon}\eta_{n},
\end{align*}
using (\ref{eq:eta_n}) in the last line, which holds here exactly
as above if we set $r^{*}=\max(r(\epsilon),(j-3)\epsilon)$ and recall
that $\epsilon>\Delta_{n}$. 
\end{proof}

\subsection{Proof of Theorem \ref{thm:CLT_II}}

For the proof set $\xi=X_0$ and replace $X$ by $X-X_0$. Since $X$ with deterministic $b$ and $\sigma$ has independent increments,  $X$ and $\xi$ are independent.  Under the stated assumptions on $b$ and $\sigma$, Assumption \aswithargsref{assu:sigmaB_local}{S_{loc}}{0}{0} holds true with $\alpha=\beta=0$, except for the boundedness of $X$ on $[0,T]$.  Even without this property Lemmas \ref{lem:estimatesForSM} and \ref{lem:estimates_f} hold true (the boundedness of $X$ was not used in the proofs).  We can now repeat the arguments in Section \ref{subsubsec:mainDecomp} and obtain the claimed result for $f\in H^1(\R^d)$ by the decomposition \eqref{eq:decomposition} and by applying Propositions \ref{prop:M_convergence_H1} and \ref{prop:E_convergence_H1}. We are left with showing $\Delta_n^{-1}D_{t,n}(f)\r{\P}0$, which follows from the next proposition. As compared to the general semimartingale case the key property for deterministic $b,\sigma$ is that the characteristic functions of the marginals $X_t$ can be computed explicitly. 
\begin{prop}
\label{prop:D_II} Suppose that $\xi$ has a bounded Lebesgue density and assume that $b,\sigma$ are deterministic càdlàg functions and that $\sup_{0\leq t\leq T}|(\sigma_{t}\sigma_{t}^{\top})^{-1}|<\infty$.  Then we have for $f\in H^{1}(\R^{d})$
and $0\leq t\leq T$
\begin{equation}
\Delta_{n}^{-1}D_{t,n}\left(f\right)\r{\P}0.
\end{equation}
\end{prop}

\begin{proof}
As in Section \ref{sec:D} apply the upper bound \eqref{eq:D} and define the $g_n^{(i)}$ this time without the supremum over $0\leq t\leq T$ as
\begin{align}
	g_n^{(i)}(u) = \Delta_n^{-2}(1+|u|^2)^{-1}\E\left[|F_{t,n}^{(i)}(u)|^{2}\right]
\end{align}
with $F_{t,n}^{(i)}(u)$ in \eqref{eq:F_t_n_1} and \eqref{eq:F_t_n_2}.  In order to conclude as after \eqref{eq:gFun} by dominated convergence, we need to show $g_n^{(i)}(u)\rightarrow 0$ as $n\rightarrow\infty$ for all $u\in\R^d$ and $\sup_{n\in\N, u\in\R^d}g_n^{(i)}(u) <\infty$.  For $g_n^{(1)}(u)$ this follows from repeating the proof of Lemma \ref{lem:F_1_bound} word for word using instead of the approximation property of $b$ in Assumption \aswithargsref{assu:sigmaB_local}{S_{loc}}{0}{0} that $b$ is càdlàg. 

Next,  consider $g_n^{(2)}(u)$.  The assumptions on $b$ and $\sigma$ imply that $\sc u{X_{r}-X_{h}}$ is independent
of $\c F_{h}$ for all $0\leq h<r\leq t$ and is $N(\int_{h}^{r}\sc u{b_{r'}}dr',\int_{h}^{r}|\sigma_{r'}^{\top}u|^{2}dr')$-distributed.  This means 
\begin{equation}
\E\left[e^{-i\sc u{X_{r}-X_{h}}}\bigg|\c F_h\right]=\E\left[e^{-i\sc u{X_{r}-X_{h}}}\right]=e^{-\frac{1}{2}\int_{h}^{r}|\sigma_{r'}^{\top}u|^{2}dr'}.\label{eq:charFun}
\end{equation}
From this we find that $4\E\left[|F_{t,n}^{(i)}(u)|^{2}\right]$ equals
\begin{align*}
& =\E\left[\left|\sum_{k=1}^{\lfloor t/\Delta_{n}\rfloor}e^{-i\sc u{X_{t_{k-1}}}}\int_{t_{k-1}}^{t_{k}}\left(t_{k}-r-\frac{\D}{2}\right)\left|\sigma_{r}^{\top}u\right|^{2}\E\bigg[e^{-i\sc u{X_{r}-X_{t_{k-1}}}}\bigg|\c F_{t_{k-1}}\bigg]dr\right|^2\right]\\
 &=\E\left[\left|\sum_{k=1}^{\lfloor t/\Delta_{n}\rfloor}e^{-i\sc u{X_{t_{k-1}}}}\int_{t_{k-1}}^{t_{k}}\left(t_{k}-r-\frac{\D}{2}\right)\left|\sigma_{r}^{\top}u\right|^{2}e^{-\frac{1}{2}\int_{t_{k-1}}^r|\sigma^\top_{r'}u|^2 dr'}dr\right|^2\right].
\end{align*}
Introduce the family of functions $\kappa(u,r)=\left|\sigma_{r}^{\top}u\right|^{2}e^{-\frac{1}{2}\int_{t_{k-1}}^r|\sigma^\top_{r'}u|^2 dr'}$ and with this $$\bar{\kappa}(k,u)=\int_{t_{k-1}}^{t_{k}}\left(t_{k}-r-\frac{\D}{2}\right)(\kappa(r,u)-\kappa(t_{k-1},u))dr.$$Since $\int_{t_{k-1}}^{t_{k}}(t_{k}-r-\D/2)dr=0$,  we find that
\begin{align*}
4\E\left[|F_{t,n}^{(i)}(u)|^{2}\right]
	&= \E\left[\left|\sum_{k=1}^{\lfloor t/\Delta_{n}\rfloor}e^{-i\sc u{X_{t_{k-1}}}}\bar{\kappa}(k,u)\right|^2\right]\\
	&=\sum_{k=1}^{\lfloor t/\Delta_{n}\rfloor}\sum_{k'=1}^{\lfloor t/\Delta_{n}\rfloor}\bar{\kappa}(k,u)\bar{\kappa}(k',u)\E\left[e^{-i\sc u{X_{t_{k-1}}-X_{t_{k'-1}}}}\right].
\end{align*}
The ellipticity of $\sigma_r\sigma^\top$ implies the existence of a constant $C_1>0$ with $\inf_{r}|\sigma_{r}^{\top}u|^{2}=\inf_{r}\sc{\sigma_{r}\sigma_{r}^{\top}u}u\geq C_1|u|^{2}$. It thus follows from \eqref{eq:charFun} that
\begin{align*}
	\E\left[e^{-i\sc u{X_{t_{k-1}}-X_{t_{k'-1}}}}\right]\leq e^{-C_1 |u|^2 |t_{k-1}-t_{k'-1}|}.
\end{align*}
Using this in the last display and upper bounding the integrand in $\bar{\kappa}(k,u)$ yields at last
\begin{align*}
	g_n^{(2)}(u) 
		&\leq \Delta_n^2|u|^2\sum_{k,k'=1}^n e^{-C_1 |u|^2 |t_{k-1}-t_{k'-1}|} \sup_{k\in\{1,\dots\},t_{k-1}\leq r\leq t_k} \frac{|\kappa(r,u)-\kappa(t_{k-1},u)|^2}{|u|^4}.
\end{align*}
As $\sigma$ is càdlàg, observe that  $|u|^{-4}\sup_{k\in\{1,\dots\},t_{k-1}\leq r\leq t_k} |\kappa(r,u)-\kappa(t_{k-1},u)|^2\rightarrow 0$ for any fixed $u\in\R^d$ as $n\rightarrow\infty$ and that
\begin{align*}
	\sup_{n\in\N,u\in\R^d}\sup_{k\in\{1,\dots\},t_{k-1}\leq r\leq t_k} \frac{|\kappa(r,u)-\kappa(t_{k-1},u)|^2}{|u|^4} <\infty,
\end{align*} 
while also
\begin{align*}
	& \Delta_n^2|u|^2\sum_{k,k'=1}^n e^{-C_1 |u|^2 |t_{k-1}-t_{k'-1}|}
		\leq  |u|^2 \int_0^T\int_0^T e^{-C_1 |u|^2 |t-t'|}dtdt'\\
	&\quad\quad \leq 2|u|^2 \int_0^T\int_{t'}^Te^{-C_1 |u|^2 (t-t')}dtdt' \lesssim |u|^2 \int_0^T e^{-C_1 |u|^2 t}dt.
\end{align*}
From this obtain  $g_n^{(i)}(u)\rightarrow 0$ as $n\rightarrow\infty$ for all $u\in\R^d$ and $\sup_{n\in\N, u\in\R^d}g_n^{(i)}(u) <\infty$, which is what we still needed to show.
\end{proof}

\subsection{Proof of Theorem \ref{thm:CLT_III}}\label{sec:CLT_III}

Recall that $X$ has independent increments as stated in the proof of Theorem \ref{thm:CLT_II}. The estimation error for $t\in [0,\Delta]$ is treated separately. Write $\Gamma_{t}(f)-\widehat{\Gamma}_{t,n}(f)=E_{0,n}+E_{1,n}$
with 
\begin{align*}
E_{0,n} =\int_{0}^{\Delta_{n}}(f(X_{r})-f(0))dr,\quad E_{1,n} =\sum_{k=2}^{\lfloor t/\Delta_{n}\rfloor}\int_{t_{k-1}}^{t_{k}}(f(X_{r})-f(X_{t_{k-1}}))dr.
\end{align*}
By a Sobolev embedding deduce for $f\in H^1(\R)$ that $f$ is $\gamma = 1/2$-Hölder continuous, that is, 
\begin{align}
	\sup_{x\neq y} \frac{|f(x)-f(y)|}{|x-y|^\gamma} < \infty.\label{eq:Holder}
\end{align}
In particular, 
\begin{align*}
	\E[|E_{0,n}|] \lesssim \Delta_n \sup_{0\leq r\leq \Delta_n} \E[|X_r|^{\gamma}],
\end{align*}
implying $E_{0,n}=o_{\P}(\Delta_{n})$.  For $E_{1,n}$,  we use the decomposition \eqref{eq:decomposition} (with sums starting at $k=2$) and aim at applying Propositions \ref{prop:M_convergence_H1} and \ref{prop:E_convergence_H1} to $M_{t,n}(f)$ and $E_{t,n}(f)$ with $Y=X$ (that is with $\xi=0$).  The respective proofs depend on $\xi$ only through applications of Lemma \ref{lem:estimates_f} and a specific argument for \eqref{eq:J2}.  

We first check that the statements (i)-(v) of Lemma \ref{lem:estimates_f} hold for $f\in H^1(\R)$, $Y=X$ and $k\geq 2$.  Part (v) of that lemma holds again by the Hölder continuity in \eqref{eq:Holder}.  The corresponding statements in parts (i)-(iv), by independence of increments,
expressions of the form $\sum_{k=2}^{n}\int_{t_{k-1}}^{t_{k}}\E[m^2_{r}(X_{t_{k-1}})]dr$ need to be bounded for certain random processes $m_{r}$. Denote the Lebesgue density of $X_{t_{k-1}}$ by $p_{t_{k-1}}$.  Due to Gaussianity and the non-degeneracy of
$\sigma$ we have $p_{t_{k-1}}\lesssim t_{k-1}^{-1/2}$ such that
\[
\int_{t_{k-1}}^{t_{k}}\E[m^2_{r}(X_{t_{k-1}})]dr\lesssim t_{k-1}^{-1/2}\int_{t_{k-1}}^{t_{k}}\int_{\R}\E[m^2_{r}(x)]dx\,dr.
\]
Using now the second inequality in Lemma \ref{lem:plancherel} the proofs of (i)-(iv) in Lemma  \ref{lem:estimates_f} provide
uniform bounds on $\sup_{0\leq r\leq T}\int_{\R}m^2_{r}(x)dx$. Since $\Delta_{n}\sum_{k=2}^{n}t_{k-1}^{-1/2}$
is summable,  (i)-(iv) of Lemma \ref{lem:estimates_f} (with $k\geq 2$) remain true.  The modification for \eqref{eq:J2} is analogous and therefore skipped.  The conclusions of Propositions \ref{prop:M_convergence_H1} and  \ref{prop:E_convergence_H1}
hold true, and we conclude by the following proposition.

\begin{prop}
\label{prop:D_III} Let $d=1$ and suppose that $b,\sigma$ are deterministic càdlàg functions and that $\inf_{0\leq t\leq T}\sigma_{t}^{2}>0$.  Then we have for $f\in H^{1}(\R)$, $\xi=0$ 
and $0\leq t\leq T$
\begin{equation*}
\Delta_{n}^{-1}D_{t,n}\left(f\right)\r{\P}0.
\end{equation*}
\end{prop}

\begin{proof}
By independence of increments write $D_{t,n}(f) = \sum_{k=2}^{\lfloor t/\Delta_{n}\rfloor} h_k(X_{t_{k-1}})$ with
\begin{align*}
h_{k}(x) & =\int_{t_{k-1}}^{t_{k}}\E[f(X_{r}-X_{t_{k-1}}+x)-f(x)\\
 & -\frac{\Delta_{n}}{2}(f(X_{t_{k}}-X_{t_{k-1}}+x)-f(x))]dr.
\end{align*}
Decompose $\E[|D_{t,n}(f)|^2]=R_{1}+R_{2}$
with $R_{1}=\sum_{k=2}^{\lfloor t/\Delta_{n}\rfloor}\E[h_{k}(X_{t_{k-1}})^{2}]$,
$R_{2}=\sum_{k\neq j,  k,j\geq 2}^{\lfloor t/\Delta_{n}\rfloor}\E[h_{k}(X_{t_{k-1}})h_{j}(X_{t_{j-1}})]$. It suffices to show $R_{1}=o(\Delta_{n}^{2})$ and $R_{2}=o(\Delta_{n}^{2})$. 

By Lemma \ref{lem:estimates_f}(iii) (which can be applied by the arguments at the beginning of this section) and the Cauchy-Schwarz inequality we have $R_1 = \sum_{k=2}^{\lfloor t/\Delta_{n}\rfloor}\E[\tilde{h}^2_k(X_{t_{k-1}})] + o(\Delta_n^2)$ with
\begin{align*}
\tilde{h}_{k}(x) 
	& =\int_{t_{k-1}}^{t_{k}}\E\left[f'(x)\left(X_r-X_{t_{k-1}}  -\frac{\Delta_{n}}{2} (X_{t_k}-X_{t_{k-1}})\right )\right ]dr \\
	& = f'(x)\E\left[\int_{t_{k-1}}^{t_{k}}(t_k-r-\frac{\Delta_{n}}{2})d X_r\right]=f'(x)\int_{t_{k-1}}^{t_{k}}(t_k-r-\frac{\Delta_{n}}{2})b_r dr,
\end{align*}
concluding by the stochastic Fubini theorem in the last line.  Hence,
\begin{align*}
  \sum_{k=2}^{\lfloor t/\Delta_{n}\rfloor}\E[\tilde{h}^2_k(X_{t_{k-1}})] &\lesssim \sum_{k=2}^{\lfloor t/\Delta_{n}\rfloor}\int_{\R}\tilde{h}^2_k(x) dx \lesssim \Delta_n^3 \norm{f}^2_{H^1},
\end{align*}
which shows $R_1 = o(\Delta_n^2)$.  

With respect to $R_{2}$ and $u\in\R$ denote by $F_{t,n,k}^{(1)}(u)$, $F_{t,n,k}^{(2)}(u)$
the summands in the sums (\ref{eq:F_t_n_1}), (\ref{eq:F_t_n_2}) from Section \ref{sec:D} such
that $|\c Fh_{k}(u)|=|\c Ff(u)|\, |F_{t,n,k}^{(1)}(u)+F_{t,n,k}^{(2)}(u)|$. Let now $p_{t_{k-1},t_{j-1}}$ denote the joint Lebesgue density of $(X_{t_{k-1}},X_{t_{j-1}})$.  Non-degeneracy of $\sigma$ and noting that $X$ is a Gaussian process yields for $k>j$ $$|\c Fp_{t_{k-1},t_{j-1}}(u,v)|\leq e^{-C_1|u|^{2}(t_{k-1}-t_{j-1})}e^{-C_1|u+v|^{2}t_{j-1}}\leq e^{-C_1|u+v|^{2}t_{j-1}}$$ for some $C_1>0$, implying uniformly in $u\in\R$
\begin{align*}
&\sum_{k\neq j, k,j\geq 2}^{n}\int_{\R}|\c Fp_{t_{k-1},t_{j-1}}(u,v)|dv\lesssim \sum_{k\neq j, k,j\geq 2}^{n} \int_{\R} e^{-C_1|u+v|^2\min(t_{j-1},t_{k-1})}dv \\
&\quad \quad \lesssim \Delta_{n}^{-2}\int_{0}^{T}\int_{\R}e^{-C|u|^{2}r}du dr \lesssim\Delta_{n}^{-2}\int_0^T r^{-1/2} dr \lesssim \Delta_n^{-2},
\end{align*}
The same result holds uniformly in $v\in\R$ when integrating with respect to $u$.  The Plancherel theorem and the Cauchy-Schwarz inequality yield
\begin{align*}
|R_{2}| &=(2\pi)^{-1}\left|\sum_{k\neq j=2}^{\lfloor t/\Delta_{n}\rfloor}\int_{\R^{2}}\c Fh_{k}(u)\c Fh_{j}(v)\c Fp_{t_{k-1},t_{j-1}}(u,v)d(u,v)\right|\\
& \lesssim\sum_{k\neq j, k,j=2}^{n}\int_{\R^2}|\c Fh_{k}(u)|^2|\c Fp_{t_{k-1},t_{j-1}}(u,v)|d(u,v)\\
 &  \lesssim \int_{\R}\left|\c Ff(u)\right|^{2}(1+|u|^2) g_n(u)du,
\end{align*}
with 
\begin{align*}
	g_n(u) = (1+|u|^2)^{-1}\sup_{k\in\{1,\dots,n\}}\Delta_n^{-2}|F_{t,n,k}^{(1)}(u)+F_{t,n,k}^{(2)}(u)|^2.
\end{align*}
It is easy to check that the upper bounds on the terms $g_n^{(i)}(u)$ in the proof of Proposition \ref{prop:D_II} yield for the summands considered here that $\sup_{n\in\N,u\in\R} g_n(u)<\infty$ and $g_n(u)\rightarrow 0$ for all $u\in\R$ as $n\rightarrow\infty$. Consequently, by dominated convergence $R_{2}=o(\Delta_{n}^{2})$. 
\end{proof}

\subsection{Proof of Theorem \ref{thm:CLT_Fs}}

By arguing as in Section \ref{subsubsec:localization} it is enough to prove the CLT for $f\in FL^s(\R^d)$ for $s\geq 1$ under Assumption \aswithargsref{assu:sigmaB_local}{S_{loc}}{\alpha}{\beta}.  We use the decomposition \eqref{eq:decomposition} and aim at applying Propositions \ref{prop:M_convergence_H1} and \ref{prop:E_convergence_H1} to $M_{t,n}(f)$ and $E_{t,n}(f)$ with $Y=X$ (that is with $\xi=0$).  The respective proofs depend on $f$ only through applications of Lemma \ref{lem:estimates_f} and a specific argument for \eqref{eq:J2}.  We first check that the statements (i)-(v) of Lemma \ref{lem:estimates_f} hold for $f\in FL^s(\R^d)$ and $Y=X$.  

From the boundedness of $X$ in Assumption \aswithargsref{assu:sigmaB_local}{S_{loc}}{\alpha}{\beta} and the embedding $FL^{1}(\R^{d})\subset C^{1}(\R^{d})$ we get $f\in C^1(\R^d)$ and that the processes $t\mapsto f(X_t)$ and $t\mapsto |\nabla f(X_t)|$ are uniformly bounded.  This and the continuity of paths of $X$ already imply (i),(iv) and (v) of Lemma \ref{lem:estimates_f}, while for (ii) we also use $\langle \nabla f(X_{t_{k-1}}),Z_r\rangle^2\lesssim |Z_r|^2$ and $\sup_{r}\E|Z_r|^2\lesssim \Delta_n$ using Lemma \ref{lem:estimatesForSM}(i,ii).  At last, (iii) is obtained from the same upper bound on $\E|Z_r|^2$ and a Taylor expansion of $f$.  On the other hand,  with $\tilde{Z}_k$ from the proof of Proposition \ref{prop:M_convergence_H1} we have for a random variable $V_k\overset{d}{\sim}N(0,I_{d})$, which is independent of $\c F_{t_{k-1}}$, and $\epsilon'>0$ that
\begin{align*}
\E[\widetilde{Z}_{k}^{2}\I_{\{|\widetilde{Z}_{k}|>\epsilon\}}|\c F_{t_{k-1}}] 
	& \lesssim \E[\Delta_{n}^{3}|V_k|^2\I_{\{\Delta_{n}^{3/2}|V_k| > \epsilon' \}}|\c F_{t_{k-1}}]\rightarrow 0,
\end{align*}
where the conclusion holds by dominated convergence. This proves \eqref{eq:J2} and the conclusions of Propositions \ref{prop:M_convergence_H1} and \ref{prop:E_convergence_H1} hold true.  We conclude the proof of the CLT by the following proposition.

\begin{prop}
\label{prop:Drift_Convergence_Fls} Let $s\geq 1$ and grant Assumption \aswithargsref{assu:sigmaB_local}{S_{loc}}{\alpha}{\beta} with $\alpha>\max(0,1-s/2)$, $\beta>0$. Then we have for $f\in FL^{s}(\R^{d})$
\[
\Delta_{n}^{-1}D_{t,n}(f)\r{ucp}\frac{1}{2}(f(X_{t})-f(X_{0}))-\frac{1}{2}\int_{0}^{t}\sc{\nabla f(X_{r})}{\sigma_{r}dW_{r}}.
\]
\end{prop}

\begin{proof}
We use the notation from Section \ref{sec:D}. Write $D_{t,n}(f)=h(0)$ with the function $h$ defined there.  By Fourier inversion $h(0)=(2\pi)^{-d}\int_{\R^{d}}\c Fh(u)du$ $\P$-almost surely such that by the triangle inequality
\begin{align*}
	\E[\sup_{0\leq t\leq T}|D_{t,n}(f)|] 
		&= \E[\sup_{0\leq t\leq T}|h(0)|]\lesssim \int_{\R^d} |\c Ff(u)| \E\left[\sup_{0\leq t\leq T}|F_{t,n}^{(1)}(u) + F_{t,n}^{(2)}(u)|\right]du\\
		& \leq \Delta_n \int_{\R^d} |\c Ff(u)| (1+|u|)^s (g_n^{(1)}(u)+g_n^{(2)}(u))^{1/2}du.
\end{align*}
The result follows from dominated convergence, using Lemmas \ref{lem:F_1_bound} and \ref{lem:F_2_bound_SM}.
\end{proof}

\subsection{The lower bound: Proof of Theorem \ref{thm:lowerBound_H1}}

The sigma field $\c G_{n}$
is generated by $X_{0}$ and the increments $X_{t_{k}}-X_{t_{k-1}}$
for $k\in\{1,\dots,n\}$. Their independence and the Markov property
imply $\E[f(X_{t})|\c G_{n}]=\E[f(X_{t})|X_{t_{k-1}},X_{t_{k}}]$,
$t_{k-1}\leq t\leq t_{k}$. By the same argument the random variables
$$Y_{k}=\int_{t_{k-1}}^{t_{k}}(f(X_{t})-\E[f(X_{t})|\c G_{n}])dt$$
are uncorrelated, implying
\[
\norm{\Gamma_{T}(f)-\E\left[\l{\Gamma_{T}(f)}\c G_{n}\right]}_{L^{2}(\P)}^{2}=\sum_{k=1}^{n}\E\left[Y_{k}^{2}\right]=\sum_{k=1}^{n}\E\left[\text{Var}_{k}\left(\int_{t_{k-1}}^{t_{k}}f(X_{t})dt\right)\right],
\]
where $\text{Var}_{k}(Z)=\E[(Z-\E[Z|X_{t_{k-1}},X_{t_k}])^2|X_{t_{k-1}},X_{t_k}]$ is the conditional variance of a random
variable $Z$ with respect to the sigma field generated by $X_{t_{k-1}}$
and $X_{t_{k}}$.  For fixed $k$ write
\begin{align*}
& \text{Var}_{k}\left(\int_{t_{k-1}}^{t_{k}}f(X_{t})dt\right)  =\text{Var}_{k}\left(\int_{t_{k-1}}^{t_{k}}(f(X_{t})-f(X_{t_{k-1}}))dt\right)=T_{k}^{(1)}+T_{k}^{(2)}+T_{k}^{(3)}\\
& \quad \text{with }T_{k}^{(1)}  =\text{Var}_{k}\left(\int_{t_{k-1}}^{t_{k}}\sc{\nabla f(X_{t_{k-1}})}{X_{t}-X_{t_{k-1}}}dt\right),\\
& \quad \quad \quad T_{k}^{(2)}  =\text{Var}_{k}\left(\int_{t_{k-1}}^{t_{k}}(f(X_{t})-f(X_{t_{k-1}})-\sc{\nabla f(X_{t_{k-1}})}{X_{t}-X_{t_{k-1}}})dt\right),
\end{align*}
and with the crossterm satisfying $|T_{k}^{(3)}|\leq2(T_{k}^{(1)}T_{k}^{(2)})^{1/2}$.  Conditional on $X_{t_{k-1}}$, $X_{t_{k}}$, the process $(X_{t})_{t_{k-1}\leq t\leq t_{k}}$
is a Brownian bridge starting from $X_{t_{k-1}}$ and ending at $X_{t_{k}}$. Hence, 
\[
\E[X_{t}-X_{t_{k-1}}|X_{t_{k-1}},X_{t_{k}}]=\frac{t-t_{k-1}}{\Delta_{n}}(X_{t_{k}}-X_{t_{k-1}}),
\]
cf. Equation 6.10 of \citep{Karatzas1991}. Write $X_{t}-X_{t_{k-1}}=\int_{t_{k-1}}^{t_{k}}\I_{\{r\leq t\}}dX_{r}$.
Then
\begin{align*}
 & \int_{t_{k-1}}^{t_{k}}(X_{t}-X_{t_{k-1}}-\E[X_{t}-X_{t_{k-1}}|X_{t_{k-1}},X_{t_{k}}])dt\\
 & =\int_{t_{k-1}}^{t_{k}}\int_{t_{k-1}}^{t_{k}}(\I_{\{r\leq t\}}-\frac{t-t_{k-1}}{\Delta_{n}})dX_{r}dt=\int_{t_{k-1}}^{t_{k}}(t_{k}-r-\frac{1}{2}\Delta_{n})dX_{r},
\end{align*}
using the stochastic Fubini theorem in the last line. From Itô's isometry and independence of increments obtain 
\begin{align*}
\E[T_{k}^{(1)}] & =\E\left[|\nabla f(X_{t_{k-1}})|^{2}\right]\int_{t_{k-1}}^{t_{k}}(t_{k}-r-\frac{1}{2}\Delta_{n})^{2}dr=\frac{\Delta_{n}^{3}}{12}\E\left[|\nabla f(X_{t_{k-1}})|^{2}\right].
\end{align*}
Recall from the proofs of Theorems \ref{thm:CLT_II} and \ref{thm:CLT_III} that the statements of Lemma \ref{lem:estimates_f} apply to $X$ with deterministic coefficients $b$ and $\sigma$ when $X_0$ has a bounded Lebesgue density or when $d=1$.  Consequently, 
\begin{align*}
	\sum_{k=1}^n \E[T_{k}^{(1)}] = \E\left[\frac{1}{12}\int_{0}^{T}|\nabla f(X_{t})|^{2}dt\right] + o(\Delta_n^2).
\end{align*}
On the other hand, the Cauchy-Schwarz inequality shows
\begin{align*}
	\E[|\sum_{k=1}^n T_k^{(2)}|] 
		&\leq \sum_{k=1}^n \E\left[ \left(\int_{t_{k-1}}^{t_{k}}(f(X_{t})-f(X_{t_{k-1}})-\sc{\nabla f(X_{t_{k-1}})}{X_{t}-X_{t_{k-1}}})dt\right)^2 \right]\\
		&\leq \Delta_n \sum_{k=1}^n \int_{t_{k-1}}^{t_{k}}\E\left[ \left(f(X_{t})-f(X_{t_{k-1}})-\sc{\nabla f(X_{t_{k-1}})}{X_{t}-X_{t_{k-1}}}\right)^2 \right]dt,
\end{align*}
which is of order $o(\Delta_n^2)$ by Lemma \ref{lem:estimates_f}(iii).  Combining the last two displays also shows $\sum_{k=1}^n T_k^{(3)}=o_\P(\Delta_n^2)$.  The result follows then from 
\begin{align*}
\Delta_{n}^{-2}\sum_{k=1}^{n}\E\left[\text{Var}_{k}\left(\int_{t_{k-1}}^{t_{k}}f(X_{t})dt\right)\right]& \rightarrow\E\left[\frac{1}{12}\int_{0}^{T}|\nabla f(X_{t})|^{2}dt\right].
\end{align*}

\section*{Acknowledgement}
The author thanks Jakub Chorowski for helpful comments on an early draft of this manuscript. 

\bibliographystyle{apalike}
\bibliography{bibliography}

\end{document}